\documentclass[12pt]{amsart}

\usepackage{amscd,amsmath,amssymb}
\usepackage{enumerate}
\usepackage{enumitem}

\newtheorem{Thm}{Theorem}[section]
\newtheorem{Prop}[Thm]{Proposition}
\newtheorem{Lemma}[Thm]{Lemma}
\newtheorem{Res}[Thm]{Result}
\newtheorem{Cor}[Thm]{Corollary}

\theoremstyle{definition}
\newtheorem{Rmk}[Thm]{Remark}
\newtheorem{Problem}[Thm]{Problem}

\numberwithin{equation}{section}

\def\Ker{\operatorname{Ker}}
\def\End{\operatorname{End}}
\def\Obj{\operatorname{Obj}}
\def\Arr{\operatorname{Arr}}
\def\sbs{\operatorname{sb}}
\def\sbbr{\operatorname{sbbr}}

\def\irl{\mathcal{L}}
\def\irll{\medspace\mathcal{L}\medspace}
\def\irr{\mathcal{R}}
\def\irrr{\medspace\mathcal{R}\medspace}
\def\irc{{\mathcal{C}}}
\def\irx{{\mathcal{X}}}
\def\irk{{\mathcal{K}}}
\def\iru{{\mathcal{U}}}
\def\irv{{\mathcal{V}}}
\def\irli{{\mathcal{LI}}}
\def\ircs{{\mathcal{CS}}}
\def\ires{{\mathcal{ES}}}

\def\sero#1#2{#1_0,#1_1,\ldots,#1_#2}
\def\word#1#2{#1_1#1_2\cdots#1_#2}

\def\id{\mathrel{\hat=}}

\def\cw{\curlywedge}

\def\lamsd#1#2{#1*_\lambda #2}

\def\hc#1#2{\irc(#1,#2)}
\def\hx#1#2{\irx(#1,#2)}

\def\al#1{\alpha(#1)}
\def\om#1{\omega(#1)}

\def\von#1{\overline{#1}}

\def\veps{\varepsilon}
\def\thetat{\widetilde{\theta}}
\def\Thetat{\widetilde{\Theta}}
\def\Xit{\widetilde{\Xi}}
\newcommand{\Iota}{\mathrm{I}}
\newcommand{\Tau}{\mathrm{T}}

\def\Av{\von A}
\def\At{\widetilde{A}}
\def\Xv{\von X}
\def\Xt{\widetilde{X}}
\def\Wt{\widetilde{W}}
\def\free2#1{{F_{\langle 2,2\rangle}\!}\left(#1\right)}

\def\Sro{(S,\rho)}

\def\act#1#2{{^{#1}\kern -2pt {#2}}}
\def\acth#1#2{{^{#1}\kern -1pt {#2}}}

\def\bfr#1#2{{\mathrm{BF}}#1(#2)}

\def\ss#1{\mathbf{s}\left(#1\right)}

\def\redto{\rightsquigarrow}

\newcommand{\balnyito}{\lfloor\kern -4pt\lfloor}
\newcommand{\jobbzaro}{\rceil\kern -4pt\rceil}
\newcommand{\bal}[1]{\balnyito#1\rfloor}
\newcommand{\jobb}[1]{\lceil#1\jobbzaro}
\newcommand{\Wr}[1]{W_{#1}^{\mathrm{right}}}
\newcommand{\Wl}[1]{W_{#1}^{\mathrm{left}}}
\newcommand{\Wre}{W^{\emptyset|\mathrm{right}}}
\newcommand{\Wle}{W^{\mathrm{left}|\emptyset}}
\newcommand{\Web}{W^{\emptyset|}}
\newcommand{\Wej}{W^{|\emptyset}}

\newcommand{\whwp}{\widehat{\wp}}
\newcommand{\ulp}{\breve{p}}
\newcommand{\ulB}{\breve{B}}
\newcommand{\ulC}{\breve{C}}
\newcommand{\ulw}{\breve{w}}

\begin{document}

\title[$E$-solid locally inverse semigroups]
{$E$-solid locally inverse semigroups} 				
\author{Tam\'as D\'ek\'any}
\address{Bolyai Institute, University of Szeged, Aradi v\'ertan\'uk tere 1, H-6720 Szeged, Hungary}
\email{tdekany@math.u-szeged.hu}
\author{M\'aria B.\ Szendrei}
\address{Bolyai Institute, University of Szeged, Aradi v\'ertan\'uk tere 1, H-6720 Szeged, Hungary}
\email{m.szendrei@math.u-szeged.hu}
\author{Istv\'an Szittyai}
\email{istvan.szittyai@gmail.com}

\date{2018.07.03}

\thanks{Research partially supported by the Ministry of 
Culture and Education (Hungary) grant no.\ FKFP 1030/1997,
and by the Hungarian National Foundation for Scientific 
Research grants no.~T37877, T48809, K60148 and K115518.}

\begin{abstract}
We prove that if $S$ is an $E$-solid locally inverse semigroup, and $\rho$ is an inverse semigroup congruence on $S$ such that the 
idempotent classes of $\rho$ are completely simple semigroups then $S$ is embeddable into a $\lambda$-semidirect product of a completely simple semigroup by $S/\rho$.
Consequently, the $E$-solid locally inverse semigroups turn out to be, up to isomorphism, the regular subsemigroups of
$\lambda$-semidirect products of completely simple semigroups by inverse semigroups.  
\end{abstract}

\maketitle

\section{Introduction}

Locally inverse semigroups form a large and important class of regular semigroups which contains several well-studied subclasses
--- above all the class of inverse semigroups and that of completely simple semigroups.  Locally inverse semigroups were introduced by Nambooripad in
\cite{12}  (under the name pseudo-inverse semigroups). The research into the structure of locally inverse semigroups was particularly
active in the late 1970's and  early 1980's, 
and several nice and deep results were established by 
McAlister, Nambooripad and Pastijn. 
For an exhaustive list of references, see \cite{4}.

The class of {$E$-solid semigroups} appeared even earlier in the structure theory of regular semigroups, see Hall~\cite{9B}. 
This wide class also contains the above mentioned prominent classes. 
Moreover, it is a common generalization of ortodox semigroups and completely regular semigroups. However, the study of the structure of {$E$-solid semigroups} 
has not been as intensive and successful as that of
locally inverse semigroups.

The study of classes of regular semigroups from universal 
algebraic point of view began in the late 1980's. 
It has turned out that these two classes are precisely those in which a theory showing close analogy to that for
usual varieties  of algebras can be developed, see 
Auinger~\cite{A}, \cite{2}, Hall~\cite{9A}, Yeh~\cite{16}, Ka\v dourek and the second author~\cite{KSz0}, \cite{KSz}, \cite{SzM}. 
For surveys, see Auinger~\cite{1}, Jones~\cite{10} and Trotter~\cite{T}. 
This progress revitalized the structure theoretical investigations in these classes, see e.g.\ Billhardt and the second author~\cite{7}.

It is proved by Billhardt and the 
third author~\cite{8} that if $S$ is an inverse semigroup and 
$\rho$ is an idempotent separating congruence on $S$ then 
$S$ is embeddable in a $\lambda$-semidirect
product of a group $K$ by $S/\rho$. 
In this paper we generalize this result for $E$-solid locally inverse semigroups.
Our main result is that if $S$ is an $E$-solid locally inverse semigroup and $\rho$ is an inverse semigroup congruence 
on $S$ such that the idempotent $\rho$-classes, as 
subsemigroups of $S$, are completely simple then
$S$ is  embeddable in a $\lambda$-semidirect product of a completely simple semigroup by $S/\rho$.
Since, by Yamada (and Hall)~\cite{15}, a regular semigroup is $E$-solid if and only if the idempotent classes of its
least inverse semigroup congruence are completely simple subsemigroups, we obtain that
the $E$-solid locally inverse semigroups are, up to isomorphism, the regular subsemigroups of
the $\lambda$-semidirect products of completely simple semigroups by inverse semigroups.
In the proof of the main result we apply the `canonical embedding technique' developed 
by Ku{\v r}il and the second author~\cite{11}
for handling embeddability of extensions by inverse semigroups in $\lambda$-semidirect products.

\section{Preliminaries}

In this section we recall the notions and summarize the results needed in the paper. For the undefined notions
and notation the reader is referred to \cite{H} and \cite{G}.

If $S$ is a regular semigroup then an {\it inverse unary operation\/} is defined to be a mapping 
$^\dagger\colon S\to S$
with the property that $s^\dagger\in V(s)$ for every $s\in S$. In particular, if $S$ is an inverse semigroup then the
unique inverse unary operation is denoted in the usual way by $^{-1}$.

Let $S$ be a semigroup, and $\irk$ a class of semigroups. 
If $\rho$ is an inverse semigroup congruence on $S$ (i.e., $S/\rho$ is an inverse semigroup) then $\rho$ is 
said to be a {\em congruence over} $\irk$ if each idempotent $\rho$-class, as a subsemigroup of $S$, belongs to $\irk$. 
In this case, the union of the idempotent $\rho$-classes, called the {\em kernel of} $\rho$ and denoted $\Ker\rho$, 
is a semilattice of semigroups in $\irk$.
Recall that if $S$ is a regular semigroup then, by Lallement's lemma, the idempotent $\rho$-classes are precisely 
the $\rho$-classes $e\rho$ where $e\in E_S$.

Let $K$ be a semigroup and $T$ an inverse semigroup. 
If $S$ is a semigroup and $\rho$ is a congruence on $S$ 
such that $S/\rho$ is isomorphic to $T$ and $\Ker\rho$ is 
isomorphic to $K$ then the pair $\Sro$ is called
an {\em extension of $K$ by $T$}. 
If, moreover, $S$ is regular then $\Sro$ is termed a {\it regular extension of $K$ by $T$\/}. 
In this case, $K$ --- being isomorphic to the kernel of an inverse semigroup congruence --- is necessarily regular. 
If $\Sro$ and $(T,\sigma)$ are extensions by inverse semigroups then an injective homomorphism
\hbox{$\psi\colon S\to T$} is defined to be an 
{\em embedding of the extension $\Sro$ into the extension $(T,\sigma)$} if the congruence induced by
$\psi\sigma^\natural$ is just $\rho$.

If $S$ is a semigroup and $K$ is a subsemigroup in $S$ then we distinguish Green's $\irr$ relation on $K$ from that on $S$ by writing
$\irr^K$. Recall that if $K$ is regular then $\irr^K=\irr\cap(K\times K)$. Moreover, if $K$ is a full regular subsemigroup in 
the regular semigroup $S$ then the rule $R\mapsto R^K=R\cap K$
determines a bijection from the set of $\irr$-classes of $S$ onto the set of $\irr$-classes of $K$.
In particular, for any $x\in K$, we have $(R_x)^K=(R^K)_x$, therefore it is
not confusing to write simply $R_x^K$.

A regular semigroup is called {\it locally inverse\/} if each local submonoid 
$eSe\ (e\in E_S)$ is an inverse semigroup. Note that this
concept is used in the literature also if $S$ is not necessarily regular. 
However, throughout this paper, we consider only
regular locally inverse semigroups (and regular $E$-solid semigroups), so we omit the attribute `regular', for short. 
It is well known that a regular semigroup $S$ is locally 
inverse if and only if the natural partial order $\le$ is 
compatible with the multiplication.

Another characteristic property of locally inverse semigroups, within the class of all regular semigroups, is that all the sandwich sets are singletons.  
This allows us to introduce another binary operation $\wedge$ on $S$, assigning to any
pair of elements $(s,t)\in S\times S$ the unique element $s\wedge t$ of the sandwich set $S(t^*t,ss^*)$, where
$s^*$ and $t^*$ are arbitrary inverses of $s$ and $t$, respectively.
We call $\wedge$ the {\em sandwich operation} on $S$. 
It is clear by definition that $s\wedge t\in E_S$ and 
$s\wedge t=ss^*\wedge t^*t$ for every $s,t\in S$ and for any 
$s^*\in V(s)$ and $t^*\in V(t)$. 
In particular, in an inverse semigroup, $s\wedge t
=ss^{-1}t^{-1}t$ and, in a completely simple semigroup, 
$s\wedge t$ is the unique idempotent which is 
$\irr$-related to $s$ and $\irl$-related to $t$, 
that is, $s\wedge t$ is the identity in the group $H_{st}$.
Let us also mention that every homomorphism and congruence of locally inverse semigroups 
respects the sandwich operation, see \cite{16}. 

The following important property of locally inverse semigroups will be also
needed later, see \cite[Proposition IX.3.2(4)]{G}.

\begin{Res}\label{li-basic}
Let $S$ be a locally inverse semigroup, and let $s,t\in S$ with
$s\le t$. Then, for every $b\in R_t$, there exists 
a unique $a\in R_s$ such that $a\le b$.
\end{Res}

A semigroup is called {\em $E$-solid} if the core of $S$, that is, the subsemigroup generated by the idempotents of $S$ is completely regular. 
Clearly, orthodox semigroups and completely regular semigroups are $E$-solid. 
It is also known, that a regular semigroup is $E$-solid if and only if the least inverse 
semigroup congruence is over the class of all completely simple semigroups, 
see Yamada (and Hall)~\cite{15}. 
Therefore every $E$-solid locally inverse semigroup can be 
viewed as an extension of a locally inverse completely 
regular semigroup (that is, of a strong semilattice of completely simple semigroups) by an inverse semigroup. 
For several equivalent characterizations of the class of strong semilattices of completely simple semigroups, see \cite{P}.

By a {\em binary semigroup} we mean a semigroup having
an additional binary operation denoted by $\wedge$. 
A homomorphism or a congruence of a binary semigroup
is always supposed to respect both the multiplication and the $\wedge$ operation. 
As noticed above, each locally inverse semigroup is also a binary semigroup with respect to the 
sandwich operation, and the homomorphisms and congruences of locally inverse semigroups, 
considered as usual semigroups and binary semigroups, respectively, coincide.

Now we recall the notion of a $\lambda$-semidirect product by an inverse semigroup introduced by Billhardt~\cite{6}
and formulate its basic properties.

Let $K$ be an arbitrary semigroup and $T$ an inverse 
semigroup. Denote by $\End (K)$ the endomorphism monoid of $K$. We say that
$T$ acts on $K$ by endomorphisms on the left if an antihomomorphism $\veps\colon T\to\End (K),\
t\mapsto\veps_t$ is given, that is, $\veps_u\veps_t=\veps_{tu}\ (t,u\in T)$ holds for the mapping $\veps$.
For brevity, we say only that $T$ {\em acts on} $K$, and we denote $a\veps_t$ by $\act ta\ (a\in K,\ t\in T)$. 
Define a multiplication on the underlying set $\{(a,t)\in K\times T:\act{tt^{-1}}a=a\}$ by
$$(a,t)(b,u)=(\act{(tu)(tu)^{-1}}a\cdot \acth tb,tu) \quad (a,b\in K,\ t,u\in T).$$
The semigroup obtained in this way is called a 
{\em $\lambda$-semidirect product of $K$ by $T$} and is denoted by $\lamsd KT$. For several reasons, the
$\lambda$-semidirect product construction can be considered as a natural generalization of a semidirect product
by a group, see \cite{13}.

By specializing the respective statements in \cite{5} and \cite{11}, we formulate the properties of this
construction in the case where $K$ is a completely simple semigroup.

\begin{Res}\label{lsdtul}
Let $K$ be a completely simple semigroup and $T$ an 
inverse semigroup acting on $K$. 
The $\lambda$-semidirect product $\lamsd KT$ is an 
$E$-solid locally inverse semigroup with set of
idempotents
$$E_{\lamsd KT}=\{(e,i):e\in E_K,\ i\in E_T
\text{ and }\act ie=e\}.$$
For any $(a,t)\in\lamsd KT$, we have
$$V_{\lamsd KT}\big((a,t)\big)=\{\big(b,t^{-1}\big):
b\in V_K(\act{t^{-1}}a) \text{ and }
\acth{t^{-1}t}b=b\}.$$
The second projection
$\pi_2\colon \lamsd KT\to T,\ (a,t)\mapsto t$
is a homomorphism of $\lamsd KT$ onto $T$, and
the congruence $\vartheta_2$ induced by $\pi_2$ is over completely simple semigroups.
The kernel of $\vartheta_2$ is
$$\Ker{\vartheta_2}=\{(a,e):a\in K,\ e\in E_T 
\text{ and }\act ea=a\},$$
and it is isomorphic to the strong semilattice of the 
completely simple subsemigroups  
$K_e=\{a\in K:\act ea=a\}\  (e\in E_T)$
of $K$ with the surjective structure homomorphisms
$\veps_f|_{K_e}\colon K_e\mapsto K_f\ (e,f\in E_T,\ e\ge f)$.
\end{Res}

The second half of the previous proposition says that 
the extension $(\lamsd KT,\vartheta_2)$ is an extension of a strong semilattice of completely simple 
subsemigroups of $K$ by $T$. 
The extension $(\lamsd KT,\vartheta_2)$ is referred to as a {\em $\lambda$-semidirect product extension of $K$ by $T$}.

In this paper, we denote by $\irli$, $\ires$ and $\ircs$
the classes of all locally inverse semigroups, 
$E$-solid semigroups and completely simple semigroups, respectively.
A class of regular semigroups is termed an {\em existence variety}, or, for short, an {\em e-variety} if it is closed 
under taking direct products, homomorphic images and 
regular subsemigroups. For example, $\irli$, $\ires$ and 
$\ircs$ form e-varieties. 
Note also that a class of inverse semigroups or a class of completely 
simple semigroups constitutes an e-variety if and only if
it forms a variety in the usual sense. 

Let $X$ be a non-empty set. The free semigroup on $X$ is denoted by $X^+$.
We `double' $X$ as follows. Consider a set 
$X'=\{x':x\in X\}$ disjoint from $X$ together with a 
bijection $'\colon X\to X',\ x\mapsto x'$, and denote 
$X\cup X'$ by $\Xv$.

Let $S$ be a regular semigroup. A mapping $\nu\colon\Xv\to S$ is called {\em matched} if $x'\nu$
is an inverse of $x\nu$ in $S$ for each $x\in X$. 
Now let $\irk$ be a class of regular semigroups. 
We say that a semigroup $B\in\irk$ together with a matched mapping $\xi\colon\Xv\to B$ is a 
{\em bifree object in $\irk$ on $X$} if, for
any $S\in\irk$ and any matched mapping $\nu\colon\Xv\to S$, there is a unique homomorphism $\phi\colon B\to
S$ extending $\nu$, that is, for which $\xi\phi=\nu$
holds. We denote the unique homomorphism extending $\nu$
by $\von\nu$. 
It was essentially proved by Yeh~\cite{16} that
an e-variety
admits a bifree object on any alphabet 
(or, equivalently, on an alphabet of at least two elements)
if and only if it is contained either in $\irli$ or in $\ires$.
The bifree objects of $\irli$ and $\ires$ are determined
by Auinger~\cite{A}, \cite{2} and by the second 
author~\cite{SzM}, respectively. 
Each of these descriptions
fit into a Birkhoff-type theory for the respective class
based on an appropriate notion of `identity', see also
\cite{KSz}. 
In this paper we need the model for the bifree objects
of the e-varieties in $\irli$ published in \cite{2}.
For a more complete introduction to the theory of 
e-varieties, see \cite{1}, \cite{10} and \cite{T}. 
Based on \cite{2}, now we summarize the concepts and results needed in the paper.

The {\em free binary semigroup $\free2 Y$ on the alphabet $Y$} can be interpreted as follows. 
Its underlying set is the smallest one among the sets $W$ which fulfil
the following conditions:
\begin{enumerate}
\item[(i)] $Y\subseteq W\subseteq \left(Y\cup\{(,\;\wedge,\;)\}\right)^+$,
\item[(ii)] $u,v\in W$ implies $uv\in W$,
\item[(iii)] $u,v\in W$ implies $(u\wedge v)\in W$.
\end{enumerate}
The operations $\cdot$ and $\wedge$ are the concatenation 
and the operation
$$\free2 Y\times\free2 Y\to \free2 Y,\ 
(u,v)\mapsto (u\wedge v)\;,$$
respectively.

One can see that the indecomposable (into a product) terms in $\free2 Y$
are precisely the elements of $Y$ called {\em letters} and the terms of the form $(u\wedge v)$.
Moreover, each term admits a unique factorization into indecomposable ones.

A {\em bi-identity} in $\irli$ is a formal equality 
$u\id v$ among terms $u,v\in\free2\Xv$. 
We say that a semigroup {\em $S\in\irli$ satisfies the bi-identity $u\id v$} if 
$u\von\nu=v\von\nu$ for each matched mapping $\nu\colon\Xv\to S$. 
The {\em bi-identity $u\id v$ holds 
in the class $\irk$ of locally inverse semigroups} if it holds in every member of $\irk$. For an e-variety $\irv$ 
of locally inverse semigroups, define
\begin{eqnarray*}
\lefteqn{\Theta(\irv,X)=\{(u,v)\in\free2\Xv\times\free2\Xv:{}}\\
&& \qquad\qquad\qquad\qquad\text{the bi-identity }u\id v\text{ holds in }\irv\}.
\end{eqnarray*}
This relation is obviously a congruence on $\free2\Xv$
which is called the {\em bi-invariant congruence on
$\free2\Xv$ corresponding to $\irv$}.
The main results of the Birkhoff-type theory for $\irli$ are the following.

\begin{Res}
A class of locally inverse semigroups forms an e-variety if and only if it is defined by a set of bi-identities.
\end{Res}

\begin{Res}\label{bifree}
Given an e-variety $\irv$ of locally inverse semigroups and a non-empty set $X$, the factor semigroup 
$\bfr \irv X=\free2\Xv/\Theta(\irv,X)$ 
together with the matched mapping 
$\xi\colon\Xv\to\bfr \irv X,\ y\mapsto y\Theta(\irv,X)$
is a bifree object in $\irv$ on $X$.
\end{Res}

Note that if $\irv$ is non-trivial, i.e., contains a member of at least to elements, then
$\xi$ is injective.
Thus the element $y\Theta(\irv,X)\ (y\in\Xv)$ is usually identified with $y$, and so 
$\bfr \irv X$ is considered to contain $\Xv$.

In the sequel we need the description, published in \cite{2}, of the 
bi-invariant congruences corresponding to the variety 
$\ircs$.

For any term $w\in\free2\Xv$, denote by $\iota w$ [$w\tau$]
the first [last] letter (i.e., element of $\Xv$) appearing in $w$ 
(reading $w$ from the left to the right as a word in the alphabet
$\Xv\cup\{(,\wedge,)\}$). 
In the usual way, extend ${}'\colon X\to X'$ to a mapping $'\colon\Xv\to \Xv$ by defining
$(x')'$ to be $x$ for any $x\in X$.

Let us consider the following {\em reductions} of the terms in $\free2\Xv$ where
$u,v\in\free2\Xv$ and $x,y,z\in\Xv$:
\begin{enumerate}
\item[(R0)] $(u\wedge v)\redto(\iota u\wedge v\tau)$,
\item[(R1)] $x(y\wedge x)\redto x$,
\item[(R2)] $(x\wedge y)x\redto x$,
\item[(R3)] $(x\wedge y)(x\wedge z)\redto (x\wedge z)$,
\item[(R4)] $(z\wedge x)(y\wedge x)\redto (z\wedge x)$,
\item[(R5)] $x'x\redto (x'\wedge x)$.
\end{enumerate}
A {\em reduction $s\redto t$ is applied} in a term $w\in\free2\Xv$ if a segment $s$ 
in $w$ is changed for $t$. A term in $\free2\Xv$ is called {\em reduced} if no reduction can be applied. 
Since reductions (R0)--(R5) shorten the terms in $\free2\Xv$, one sees that a reduced form 
can be obtained for any term by applying finitely many reductions.
In \cite{2}, each term $w\in\free2\Xv$ is proved to have a uniquely 
determined reduced form which is denoted by $\ss w$, and the following result is obtained. 

{\begin{Res}\label{Thetacs}
For any non-empty set $X$, we have
$$\Theta(\ircs,X)=\{(u,v)\in\free2\Xv\times\free2\Xv:\ss u=\ss v\}.$$
\end{Res}

Notice that, applying reduction (R0) for any indecomposable factor of a term in $\free2\Xv$
which is not a letter, we obtain an element of the free semigroup $\Xt^+$ on the alphabet 
$\Xt=\Xv\cup(\Xv\wedge\Xv)$ where $(\Xv\wedge\Xv)$ stands for the set
$\{(x\wedge y):x,y\in\Xv\}$.
In particular, every reduced term belongs to $\Xt^+$.
Thus the model of a bifree object in $\ircs$ on $X$ provided by Results \ref{bifree} and \ref{Thetacs}
can be simplified as follows.
Make the free semigroup $\Xt^+$ to a binary semigroup by defining an additional binary operation
$\wedge$ on it by
\begin{equation}\label{cswedge}
(u\wedge v) = (\iota u\wedge v\tau)
\end{equation}
for every $u,v\in\Xt^+$,
and consider the restriction of $\Theta(\ircs,X)$ to $\Xt^+$.

\begin{Prop}\label{Thetacsmod}
For any non-empty set $X$, the relation
$$\Thetat(\ircs,X)=\{(u,v)\in\Xt^+\times\Xt^+:\ss u=\ss v\}$$
is a congruence on the binary semigroup $\Xt^+$, and the factor semigroup
$\Xt^+/\Thetat(\ircs,X)$ 
together with the matched mapping 
$\xi\colon\Xv\to \Xt^+/\Thetat(\ircs,X)$, $y\mapsto y\Thetat(\ircs,X)$
is a bifree object in $\ircs$ on $X$.
\end{Prop}

Later on, we use the latter model for the bifree object of $\ircs$ on $X$, and
need an alternative description of $\Thetat(\ircs,X)$.
In order to distinguish the elements of the two types in the alphabet $\Xt$, we call the elements of $\Xv$, 
as usual, {\em letters}, and those of $(\Xv\wedge\Xv)$ {\em $\wedge$-letters}. 

\begin{Lemma}\label{Thetatcs}
The congruence $\Thetat(\ircs,X)$ is generated on $\Xt^+$, as a semigroup congruence, by the relation
$\mathrm{I}\cup \Upsilon$ where
$$\mathrm{I}=\{(xx'x, x): x\in \Xv\},$$
and $\Upsilon$ is the union of the following three relations coming from rules {\rm (R3)}--{\rm (R5)}:
\begin{eqnarray*}
\Upsilon_3 &\!\!\!=\!\!\!& \{\big((x\wedge y)(x\wedge z),(x\wedge z)\big): x,y,z\in\Xv\},\cr
\Upsilon_4 &\!\!\!=\!\!\!& \{\big((z\wedge x)(y\wedge x),(z\wedge x)\big): x,y,z\in\Xv\},\cr
\Upsilon_5 &\!\!\!=\!\!\!& \{\big(x'x,(x'\wedge x)\big): x\in\Xv\}.
\end{eqnarray*}
\end{Lemma}

\begin{proof}
Denote the semigroup congruence defined in the lemma by $\chi$.
It is obvious by Proposition \ref{Thetacsmod} that $\chi$ is contained in
$\Thetat(\ircs,X)$, and that, in order to show the reverse inclusion, it suffices to verify that the pairs 
$$\big(x(y\wedge x),x\big)\ \hbox{and}\ \big((x\wedge y)x,x\big)\quad (x,y\in\Xv),$$
coming from rules (R1)--(R2), belong to $\chi$.
Indeed, applying the relations $\mathrm{I},\,\Upsilon_5,\,\Upsilon_4,\,\Upsilon_5,\,\mathrm{I}$, we see that
$$x(y\wedge x)\ \chi\ xx'x(y\wedge x)\ \chi\ x(x'\wedge x)(y\wedge x)
\ \chi\ x(x'\wedge x)\ \chi\ xx'x\ \chi\ x$$
for every $x,y\in\Xv$.
The statement for the pairs of the other type is proven dually.
\end{proof}

Now we summarize the basic notions on graphs and semigroupoids needed in the paper.

A {\em graph} $\irx$ consists of a set of {\em objects} denoted by $\Obj\irx$ and, for every pair
$g,h\in\Obj\irx$, a set of {\em arrows} from $g$ to $h$ which is denoted by $\hx gh$. 
The sets of arrows corresponding to different pairs of
objects are supposed to be disjoint, and the set of all arrows is denoted by $\Arr\irx$. 
If $a\in\hx gh$ then 
we write that $\al a=g$ and $\om a=h$.
By a {\em loop} we mean an arrow $a$ with $\al a=\om a$. 
The arrows $a,b\in\Arr\irx$ are called {\em coterminal} 
if $\al a=\al b$ and $\om a=\om b$, and are termed 
{\em consecutive} if $\om a=\al b$.

A {\em semigroupoid} is a graph $\irx$ equipped with a composition which assigns to every pair of consecutive
arrows $a\in\hx gh,\ b\in\hx hi$ an arrow in $\hx gi$, usually denoted by $a\circ b$, such that the composition is associative, that
is, for any arrows $a\in\hx gh,\ b\in\hx hi$ and $c\in\hx ij$, we have $(a\circ b)\circ c=a\circ (b\circ c)$.

The notion of a semigroupoid generalizes that of a semigroup. Indeed, each semigroup can be considered as the set of arrows
of a semigroupoid whose set of objects is a singleton. A number of basic notions in semigroup theory can be extended in a natural way
for semigroupoids.

Let $\irx$ be a semigroupoid. We define {\em Green's relation $\irl$ on $\irx$} as follows: for any $a,b\in\Arr\irx$, we have
$a\,\irl\,b$ if and only if either $a=b$, or there exist $u,v\in\Arr\irx$ such that $\om u=\al a,\ \om v=\al b$ and $u\circ a=b,\ v\circ b=a$.
It is routine to check that $\irl$ is an equivalence relation on $\Arr\irx$, and 
clearly, for any $a,b\in\Arr\irx$ with $a\,\irl\,b$, we have $\om a=\om b$.
Furthermore, for any $c\in\Arr\irx$ with $\al c=\om a=\om b$, the relation $a\,\irl\,b$ implies
$a\circ c\,\irl\,b\circ c$.
Dually, we can introduce Green's relation $\irr$ on $\irx$ and formulate its basic properties.

By a {\em regular semigroupoid} we mean a semigroupoid $\irx$ in which, for every arrow $a\in\hx gh$, 
there exists an arrow $b\in\hx hg$ with $a\circ b\circ a=a$. 
If the arrows $a\in\hx gh$, $b\in\hx hg$ 
have the property that
$a\circ b\circ a=a$ and $b\circ a\circ b=b$ then we call 
$b$ an {\em inverse} of $a$, and denote the set of all inverses of $a$ by $V(a)$. 
Similarly to a regular semigroup, each arrow of $\irx$ has an inverse, and
each $\irr$- and $\irl$-class contains an idempotent arrow.
Each idempotent arrow is a loop, i.e., belongs to $\hx gg$ for some $g\in\Obj\irx$, and
$\hx gg$ is a regular semigroup for every $g\in\Obj\irx$.
Therefore the notion of the sandwich set $S(e,f)$ is defined for every $g\in\Obj\irx$ and $e,f\in E(\hx gg)$.
These sandwich sets are singletons if and only if $E(\hx gg)$ is a locally inverse semigroup for any $g\in\Obj\irx$.
If $\irx$ is a regular semigroupoid with this property then we call it a {\em locally inverse semigroupoid}, and
we define a sandwich operation on it as follows: if $a,b\in\Arr\irx$ such that $\al a=\om b$ then
$(a\cw b)$ is the unique element of $S(b'b,aa')$ where $a'\in V(a),\ b'\in V(b)$.
Note that $(a\cw b)$ is independent of the choice of $a',b'$, and $(a\cw b)\in E\big(\irx(\al a,\al a)\big)$.
For completeness, let us mention that also the notion of a {\em natural partial order} can be introduced for regular semigroupoids so that its properties are similar to those well known for regular semigroups.
In particular, the natural partial order of $\irx$ is compatible with $\circ$ if and only if $\irx$ is locally inverse,
and in this case, it is compatible also with $\cw$.

Motivated by the notion of a locally inverse semigroupoid which is a semigroupoid with an additional partial binary operation, now we introduce a notion of a {\em free binary semigroupoid $\free2\irx$ on a graph $\irx$}.
Consider the smallest set $P$ such that it is the disjoint union of its subsets $P_{g,h}\ (g,h\in\Obj\irx)$, and
the following conditions are satisfied:
\begin{enumerate}
\item[(i)] $\hx gh\subseteq P_{g,h}\subseteq \big(\!\Arr\irx\cup\{(,\;\wedge,\;)\}\big)^+$,
\item[(ii)] $p\in P_{g,h},\ q\in P_{h,k}$ imply $pq\in P_{g,k}$,
\item[(iii)] $p\in P_{g,h},\ q\in P_{k,g}$ imply $(p\wedge q)\in P_{g,g}$.
\end{enumerate}
The elements of $P$ are called {\em paths} (more precisely, {\em binary paths}), and if $p\in P_{g,h}$ then
we define $\al p=g$ and $\om p=h$.
The free binary semigroupoid $\free2\irx$ on $\irx$ is defined as follows: its set of objects and arrow sets are $\Obj\free2\irx=\Obj\irx$ and  $\free2\irx(g,h)=P_{g,h}\ (g,h\in \Obj\free2\irx)$, respectively, and
the operations are those in (ii) and (iii) above.
Notice that $\Arr\free2\irx\subseteq \free2{\Arr\irx}$.

Let us `double' $\irx$ as follows. Consider a graph $\irx '$ such that 
$\Obj{\irx '}=\Obj\irx$, the set $\Arr{\irx '}$ is disjoint 
from $\Arr\irx$, and a bijection 
$${}'\colon\hx gh\to \irx '(h,g),\quad a\mapsto a'$$
is fixed for every $g,h\in \Obj{\irx '}$.
Define the graph $\von{\irx}$ by $\Obj{\von{\irx}}=\Obj\irx$ and 
$\von{\irx}(g,h)=\hx gh\cup \irx '(g,h)\ (g,h\in \Obj{\von{\irx}})$.
Notice that the bijections $'$ from the arrow sets of $\irx $ onto those of $\irx '$ 
determine a bijection from $\Arr\irx$ onto $\Arr{\irx '}$. 
Therefore, $\Arr\irx\cup\Arr{\irx '}$ is a doubling of the set $\Arr\irx$.
For brevity, put $A=\Arr\irx$, and assume that $A'=\Arr\irx '$.
Thus $\Av=\Arr\von{\irx}$ also follows, and 
each (binary) path in $\von{\irx}$ can be also considered as a term in $\free2\Av$.
In particular, a $\wedge$-letter $(a\wedge b)\in (\Av\wedge\Av)$ is a path in $\von{\irx}$ if and only if
$\al a=\om b$. 
Such a $\wedge$-letter will be termed a {\em $\wedge$-loop}.
Obviously, a word $\word am\in \At^+$ is a path in $\von{\irx}$ if and only if $a_i$ is either a letter or 
a $\wedge$-loop for every $i\ (1\le i\le m)$, and 
$\om{a_1}=\al{a_2},\ \om{a_2}=\al{a_3},\ \ldots,\ \om{a_{m-1}}=\al{a_m}$.
It is straightforward that the subgraph of $\free2{\von{\irx}}$ whose arrows are just the (binary) paths in
$\von{\irx}$ belonging to $\At^+$ forms a subsemigroupoid.
This subsemigroupoid is denoted $\widetilde{\irx}^+$.
Equality \eqref{cswedge}, applied only for $u,v\in\Arr{\widetilde{\irx}^+}$ with $\al u=\om v$, defines a
$\wedge$ operation on $\widetilde{\irx}^+$ so that $\widetilde{\irx}^+$ can be also considered a binary semigroupoid.

Given a regular semigroupoid $\irx$, a transformation $^\dagger\colon\Arr\irx\to\Arr\irx$ is called 
an {\em inverse unary operation} on $\irx$ if $a^\dagger\in V(a)$ for any arrow $a$.

Let $\irx$ be a semigroupoid and $S$ a semigroup. If 
$\ell\colon\irx\to S$ is a morphism of semigroupoids, i.e.,
$\ell(a\circ b)=\ell(a)\cdot\ell(b)$ for any
pair of consecutive arrows $a,b$ in $\irx$
then $\ell$ is said to be a {\em labelling of $\irx$ by $S$}.
Note that if both $\irx$ and $S$ are locally inverse then $\ell$ is also a binary morphism.

\section{Main result}

The aim of the paper is to prove the following embedding theorem.

\begin{Thm}\label{main}
Let $S$ be an $E$-solid locally inverse semigroup and $\rho$ an inverse semigroup congruence on $S$
such that the idempotent classes of $\rho$ are completely simple subsemigroups in $S$.
Then the extension $\Sro$ can be embedded into a $\lambda$-semidirect product extension of a
completely simple semigroup by $S/\rho$.
\end{Thm}

Recall that, in an $E$-solid semigroup, the idempotent congruence classes of the least inverse semigroup
congruence are completely simple subsemigroups, see \cite{15}. 
Taking into account Result~\ref{lsdtul} and that both classes of $E$-solid and of locally inverse semigroups are closed under taking regular subsemigroups,
we immediately deduce the following characterization of $E$-solid locally inverse semigroups.

\begin{Cor}\label{cor32}
A regular semigroup is $E$-solid and locally inverse if and only if it is embeddable in a $\lambda$-semidirect product
of a completely simple semigroup by an inverse semigroup.
\end{Cor}

In particular, this statement provides a structure theorem that constructs $E$-solid locally inverse semigroups from
completely simple and inverse semigroups by means of two fairly simple constructions: forming $\lambda$-semidirect product
and taking regular subsemigroup.

Note that, when restricting our attention to inverse semigroups, the extensions considered in Theorem~\ref{main} are 
just the idempotent separating extensions. 
Thus the following weaker version of the main result of \cite{8} easily follows from Theorem~\ref{main}.

\begin{Cor}
If $S$ is an inverse semigroup and $\rho$ an idempotent separating congruence on $S$
then the extension $\Sro$ can be embedded into a $\lambda$-semidirect product extension of a
group by $S/\rho$.
\end{Cor}

\section{Construction}

In this section the canonical construction of \cite{11} is adapted to derive an embeddability
criterion for the extensions considered in the paper.

Throughout this section, 
let $\Sro$ be an extension by an inverse semigroup where $S$ is an $E$-solid locally inverse
semigroup and $\rho$ is over the class $\ircs$ of all completely simple semigroups.
For brevity, denote the factor semigroup $S/\rho$ by $T$ and its elements by lower 
case Greek letters. 
Recall that $\Ker\rho$ is a strong semilattice of completely simple semigroups. 

By making use of Result~\ref{li-basic}, it is routine to 
extend a well-known property of strong semilattices of completely simple semigroups
(cf.\ \cite[Lemma II.4.6(ii) and Theorem IV.1.6(iii),(iv)]{P}) to $E$-solid locally inverse semigroups as follows.

\begin{Lemma}\label{Sro-basic}
For every $\alpha,\beta\in T$ with $\alpha\ge\beta$, and for every $s\in\alpha$, 
there exists a unique $t\in\beta$ such that $s\ge t$.
\end{Lemma}

\begin{proof}
Recall that $K=\Ker\rho$ is a full regular subsemigroup in $S$, and so the rule $R\mapsto R^K=R\cap K$
determines a bijection from the set of $\irr$-classes of $S$ onto the set of $\irr$-classes of $K$.
Let $\alpha,\beta\in T$ with $\alpha\ge \beta$, and let $s\in\alpha$, $s'\in V(s)$. 
Then $(ss')\rho=s\rho (s\rho)^{-1}=\alpha\alpha^{-1}\ge\beta\beta^{-1}$, and
$R_s^K=R_{ss'}^K$ is an $\irr$-class in the completely simple subsemigroup $\alpha\alpha^{-1}$ of $K$.
Since $K$ is a strong semilattice of the completely simple semigroups $\epsilon\ (\epsilon\in E(T))$,
there is a unique $\irr$-class $R^K$ of the completely simple semigroup $\beta\beta^{-1}$ such that
$R^K\le R_s^K$. In fact, $R^K$ is the $\irr$-class of $\beta\beta^{-1}$ containing the unique idempotent
$e$ of $K$ (or, equivalently, of $S$) such that $e\in\beta\beta^{-1}$ and $e\le ss'$.
The inequality $R^K\le R_s^K$ implies $R\le R_s$, $R$ being the $\irr$-class of $S$ containing $R^K$.
Finally, we deduce by Result~\ref{li-basic} that there is a unique $t\in R$ with $t\le s$. This proves
the existence of $t$. 
Uniqueness also follows if we observe that $t\le s$ in $S$ and 
$t\rho (t\rho)^{-1}=\beta\beta^{-1}$ imply that, for any $t'\in V(t)$, we have
$R_{tt'}=R_t\le R_s=R_{ss'}$, whence $R_{tt'}^K\le R_{ss'}^K$ follows. Furthermore, $R_{tt'}^K$ is
an $\irr$-class of the completely simple semigroup $\beta\beta^{-1}$, and so $R_{tt'}^K=R^K$. 
\end{proof}

Now we recall the canonical construction of \cite{11} and adapt it to our purposes.

First we define the {\em derived semigroupoid} $\irc$ corresponding to the extension $\Sro$ as follows.
Let $\Obj\irc=T$
and, for any $\alpha,\beta\in T$, let
$$\hc\alpha\beta= \bigl\{ (\alpha,s,\beta)\in T\times S\times T : 
\alpha\cdot s\rho=\beta\hbox{ and }\beta\cdot (s\rho)^{-1}=\alpha \bigr\}.$$
Therefore 
$\al a=\alpha$ and $\om a=\beta$ for any arrow $a=(\alpha,s,\beta)\in\Arr\irc$.
Composition is defined in $\irc$ in the following manner:
if $(\alpha,s,\beta)\in\hc\alpha\beta$ and
$(\beta,t,\gamma)\in\hc\beta\gamma$
then
$$ (\alpha,s,\beta)\circ (\beta,t,\gamma)=(\alpha,st,\gamma). $$
Clearly, this operation is associative, and so $\irc$ forms a
semigroupoid. 
Furthermore, by putting $\ell(a)=s$ for every arrow
$a=(\alpha,s,\beta)\in\Arr\irc$,
we define a labelling of $\irc$ by $S$.
Since $S$ is a regular semigroup, $\irc$ is a
regular semigroupoid, in which
$V\big( (\alpha,s,\beta) \big)=
\{ (\beta,s^*,\alpha):s^*\in V(s) \}$.
Hence, for every $\alpha\in T$, the semigroup
$\irc(\alpha, \alpha)$ 
is easily seen to be regular and isomorphic to a subsemigroup
of the locally inverse semigroup $S$. 
Therefore $\irc$ is a locally inverse semigroupoid,
and so the sandwich operation $\cw$ is also defined, and the natural partial order of
$\irc$, where
$a\le b$ for any $a,b\in\Arr\irc$ if and only if $\al a=\al b$, $\om a=\om b$
and $\ell(a)\le\ell(b)$,
is compatible both with $\circ$ and $\cw$.

Consider the graphs $\irc '$ and $\von{\irc}$ corresponding to $\irc$, and put
$A=\Arr\irc$, $A'=\Arr\irc '$.
Then we have $\Av=A\cup A'=\Arr\von{\irc}$.

Let us choose and fix an inverse unary operation $^\dagger$ on $S$. 
This determines an inverse unary operation,
also denoted by $^\dagger$, on $\irc$ by letting $(\alpha,s,\beta)^\dagger=(\beta,s^\dagger,\alpha)$ 
for every
$(\alpha,s,\beta)\in \Arr\irc$.
Consider the congruence $\theta$ on the free 
binary semigroup $\free2\Av$ generated by
$$\Theta(\ircs,A) \cup \Xi_1 \cup \Xi_2$$
where $\Theta(\ircs,A)$ is the bi-invariant congruence on 
$\free2\Av$ corresponding to $\ircs$ (see Result \ref{Thetacs}), and
\begin{eqnarray*}
\Xi_1&=&\left\{(a',a^\dagger) : a\in A\right\},\\
\Xi_2&=&\left\{(ab,c) : a,b,c\in A\hbox{ and } 
a\circ b=c\ \hbox{ in }\irc\right\}.
\end{eqnarray*}
The factor semigroup $\free2\Av/\theta$ is clearly isomorphic to the factor semigroup of $\bfr \ircs A$
over the congruence generated by $\Xi_1\cup\Xi_2$.
In \cite{11}, this is the first factor of the $\lambda$-semidirect product constructed to embed the extension 
$\Sro$ into.
Moreover, it is proved to be independent, up to isomorphism, of the choice of the inverse unary operation of $S$ involved in the construction.

Now we apply the idea of replacing each term of $\free2\Av$ by the word belonging to $\At^+$ which is obtained from it 
by applications of (R0). 
Result \ref{Thetacs}, Proposition \ref{Thetacsmod} and Lemma \ref{Thetatcs} imply that the completely simple semigroup $\free2\Av/\theta$ is isomorphic
to $\At^+/\thetat$ where $\thetat$ is the semigroup congruence generated by
$\mathrm{I}\cup \Upsilon\cup \Xit_1\cup \Xit_{21}\cup \Xit_{22}$
where
\begin{eqnarray*}
\Xit_1\!\!\!&=&\!\!\!\Xi_1,\\
\Xit_{21}
\!\!\!&=&\!\!\!
\left\{(ab,c) : a,b,c\in A,\hbox{ either }a\hbox{ and }b\hbox{ or }c\hbox{ are letter factors,}\right.\\
\!\!\!&\ &\!\!\!\quad \left.\hbox{ and } 
a\circ b=c\ \hbox{ in }\irc\right\},\\
\Xit_{22}
\!\!\!&=&\!\!\!
\left\{\big((a\wedge y),(c\wedge y)\big) : a,c\in A,\ y\in\Av,\hbox{ and } 
a\circ b=c\ \hbox{ in }\irc\right.\\
\!\!\!&\ &\!\!\!\quad \left.\hbox{ for some } b\in A\right\}\\
\!\!\!&\cup&\!\!\!
\left\{\big((y\wedge b),(y\wedge c)\big) : b,c\in A,\ y\in\Av,\hbox{ and } 
a\circ b=c\ \hbox{ in }\irc\right.\\
\!\!\!&\ &\!\!\!\quad \left.\hbox{ for some } a\in A\right\}.
\end{eqnarray*}

By the well-known description of a semigroup congruence generated by a relation, 
we easily deduce the following lemma.

\begin{Lemma}\label{leirtaut}
Let $u,v$ be words in $\At^+$. Then $(u,v)\in \thetat$ if and only if there exists a finite sequence of
words $u=\sero wn=v$ in $\At^+$  such that, for any $i$ $(i=0,1,\ldots,n-1)$, the word $w_{i+1}$
is obtained from $w_i$ by one of the following steps:
\begin{enumerate}
\item[{\rm (S$j$a)}] replacing a section $s$ of $w_i$ by $t$,
\item[{\rm (S$j$b)}] replacing a section $t$ of $w_i$ by $s$
\end{enumerate}
where $j=1,21,22$, $(s,t)\in\Xit_j$, and 
\begin{enumerate}
\item[{\rm (T$j$a)}] replacing a section $s$ of $w_i$ by $t$,
\item[{\rm (T$j$b)}] replacing a section $t$ of $w_i$ by $s$,
\end{enumerate}
where $j=3,4,5$, $(s,t)\in\Upsilon_j$. 
\end{Lemma}

Put $K=\At^+/\thetat$. Since 
$\Thetat(\ircs,A)\subseteq\thetat$, we have
$K\in\ircs$.
The equality
$$ \act\pi{(\alpha,s,\beta)}=(\pi\alpha,s,\pi\beta)  \qquad  (\pi\in T,\ (\alpha,s,\beta)\in A)$$
defines an action of $T$ on the semigroupoid $\irc$ in the sense that the following properties hold:
$\act\pi{(a\circ b)}=\act\pi a\circ\acth\pi b$
for any $\pi\in T$ and any arrows $a,b\in A$ with $\om a=\al b$,
and also $\act\pi{(\act\nu{a})}=\act{\pi\nu}a$ for every $\pi,\nu\in T$ and $a\in A$. 
Note that $\act\pi{(a\cw b)}=(\act\pi a\cw\acth\pi b)$ also holds
for any $\pi\in T$ and any arrows $a,b\in A$ with $\al a=\om b$
since morphisms of locally inverse semigroupoids respect $\cw$.
Moreover, it is also clear
that $\act\pi{(a^\dagger)}=(\act\pi{a})^\dag$ 
for every $a\in A$ and $\pi\in T$.
This ensures that, for every $\pi\in T$, the mapping
$A\to A,\ a\mapsto \act\pi a$
can be naturally extended to an endomorphism
$\veps_\pi$ of $K$ such that 
$\veps_\varsigma\veps_\pi=\veps_{\pi\varsigma}$ holds for
every $\pi,\varsigma\in T$.
Therefore $\pi\mapsto\veps_\pi$ defines an action $\veps$ of $T$ on $K$
by the rule $\act\pi{(u\thetat)}=(\act\pi u)\thetat\ (\pi\in T$,\ $u\in \At^+)$
where $\act\pi u$ is the word obtained from $u$ by replacing each letter $a$ of $u$ by $\act\pi a$.

Consider the $\lambda$-semidirect product $\lamsd{K}{T}$ determined by this action, and define a mapping
$$
\kappa\colon S\to\lamsd{K}{T}
\quad\hbox{by}\quad
s\mapsto\bigl((s\rho(s\rho)^{-1},s,s\rho)\thetat,s\rho\bigr).
$$
It is easily seen that $\kappa$ is a homomorphism, and the congruence induced by $\kappa\pi_2$ is just $\rho$. 
Furthermore, the following important property of $\kappa$ is implied by the main result of \cite{11}:

\begin{Res}\label{crit}
Let $S$ be an $E$-solid locally inverse semigroup, and let 
$\rho$ be an inverse semigroup congruence on $S$ over $\ircs$.
Then the extension $\Sro$ is embeddable in a
$\lambda$-semidirect product extension of a completely simple semigroup by an inverse semigroup if and only if $\kappa$ is an
embedding,  or, equivalently, if and only if the relations   $s\,\rho\,t$ in $S$ and
$(s\rho(s\rho)^{-1},s,s\rho)\,\thetat\,(t\rho(t\rho)^{-1},t,t\rho)$ in $\At^+$ imply $s=t$ for every $s,t\in S$.
\end{Res}

In the next section we apply this result to prove our main result Theorem \ref{main}.
By Lemma~\ref{leirtaut}, this is based on the study of the
a finite sequences
$u=\sero wn=v$ of words in $\At^+$ where, for every $i$ $(i=0,1,\ldots,n-1)$, the word $w_{i+1}$
is obtained from $w_i$ by one of the steps
{\rm (S$j$a), (S$j$b) with $j=1, 21, 22$ and (T$j$a), (T$j$b) with $j=3,4,5$.
Later on, such a sequence $u=\sero wn=v$ is called a 
$\ircs$-{\em derivation from $u$ to $v$ in $\At^+$}.

\section{Proof of Theorem \ref{main}}

This section is devoted to proving that the homomorphism $\kappa$ introduced in the previous section is,
indeed, an embedding.

Let $\Sro$ be an extension where
$S$ is an $E$-solid locally inverse semigroup and $\rho$ is an inverse semigroup congruence on $S$
over $\ircs$. Consider the construction --- in particular, 
$T$, $\irc$, $\von{\irc}$, $A$, $\Av$, $\At^+$, ${}^\dagger$, $\thetat$, 
$K$ and $\kappa$ ---
corresponding to $\Sro$ as introduced in the previous section. 

Now we observe that adjacency in the semigroupoid $\irc$ is closely
related to Green's relation $\irr$ in $T$, and that there is a crucial 
connection between the endpoints and the
label of an arrow.

\begin{Lemma}\label{tul}
Let $(\alpha,s,\beta)\in T\times S\times T$.
\begin{enumerate}
\item\label{tul1}
We have $(\alpha, s, \beta)\in\Arr\irc$ if and only if  $\alpha\irrr\beta$ in $T$ and $s\rho\ge\alpha^{-1}\beta$.
\item\label{tul2}
If $(\alpha,s,\beta)\in\Arr\irc$ then  
$s\rho (s\rho)^{-1}\geq\alpha^{-1}\alpha$ and
$(s\rho)^{-1} s\rho\geq\beta^{-1}\beta$ in $T$.
\item\label{tul3}
If $(\alpha,s,\beta)\in\Arr\irc$ then the following properties are
equivalent:
\begin{enumerate}
\item[(a)] $s\rho = \alpha^{-1}\beta$,
\item[(b)] $s\rho (s\rho)^{-1} = \alpha^{-1}\alpha$,
\item[(c)] $(s\rho)^{-1} s\rho = \beta^{-1}\beta$.
\end{enumerate}
\end{enumerate}
\end{Lemma}

\begin{proof}
(\ref{tul1})\quad
Let $(\alpha,s,\beta)\in\Arr\irc$; then $\alpha\cdot s\rho=\beta$ and
$\beta\cdot(s\rho)^{-1}=\alpha$ in $T$. 
Hence we deduce $\alpha\irrr\beta$ and
$\alpha^{-1}\alpha\cdot s\rho=\alpha^{-1}\beta$, 
and so $s\rho\ge \alpha^{-1}\beta$ is implied in the inverse semigroup $T$. 
Now let $\alpha$ and $\beta$ be $\irr$-related elements in $T$, that is, let
$\alpha\alpha^{-1}=\beta\beta^{-1}$, and let $s\in S$ with $s\rho\ge
\alpha^{-1}\beta$.
We clearly have 
$\alpha\cdot s\rho\ge \alpha\alpha^{-1}\beta=\beta\beta^{-1}\beta=\beta$ and
$\beta\cdot (s\rho)^{-1}\ge \beta (\alpha^{-1}\beta)^{-1}=
\beta\beta^{-1}\alpha = \alpha\alpha^{-1}\alpha=\alpha$,
whence
$\alpha\ge\alpha\cdot s\rho (s\rho)^{-1}\ge \beta\cdot (s\rho)^{-1}$
and
$\beta\ge \beta\cdot (s\rho)^{-1} s\rho\ge \alpha\cdot s\rho$
follow, respectively.
So, by the definition of $\irc$, we see that $(\alpha,s,\beta)\in\Arr\irc$.

(\ref{tul2})\quad
If $(\alpha,s,\beta)\in\Arr\irc$ then (1) implies
$s\rho (s\rho)^{-1}\ge \alpha^{-1}\beta (\alpha^{-1}\beta)^{-1}=
\alpha^{-1}\beta\beta^{-1}\alpha=\alpha^{-1}\alpha\alpha^{-1}\alpha=
\alpha^{-1}\alpha$, and similarly, 
$(s\rho)^{-1} s\rho\ge \beta^{-1}\beta$.

(\ref{tul3})\quad
Straightforward.
\end{proof}

Lemma \ref{tul}(\ref{tul2}),(\ref{tul3}) indicate that 
some arrows in $\irc$ 
are special in the sense that their labels are `as low as
possible'. An arrow $(\alpha,s,\beta)$ in $\irc$ having the property that 
$s\rho=\alpha^{-1}\beta$ (cf.~(3)) is termed {\em stable}. 
Consider the subgraph $\widehat\irc$ of $\irc$ where $\Obj{\widehat\irc}=\Obj\irc$ and
$\Arr{\widehat\irc}$ consists of all stable arrows of $\irc$.

\begin{Lemma}\label{ujall}
Let $a,b$ be consecutive arrows in $\irc$.
\begin{enumerate}
\item \label{ujall1}
If $a\in\Arr{\widehat\irc}$ then each inverse of $a$ is in 
$\Arr{\widehat\irc}$. 
\item \label{ujall2}
If $a\in\Arr{\widehat\irc}$ then 
$a\circ b\in\Arr{\widehat\irc}$ and 
$a\irrr a\circ b$.
\item \label{ujall3}
If $a\in\Arr{\widehat\irc}$ then 
$(b\cw a)\in\Arr{\widehat\irc}$ and 
$a\irll (b\cw a)$.
\end{enumerate}
\end{Lemma}

\begin{proof}
(\ref{ujall1})\quad
Clear by definition.

(\ref{ujall2})\quad
Suppose that $a=(\alpha, s, \beta)\in\Arr{\widehat\irc}$ and 
$b=(\beta, t, \gamma)\in\Arr\irc$. 
Then $(s\rho)^{-1}s\rho=\beta^{-1}\beta$ by Lemma~\ref{tul}(\ref{tul3}), and
$\beta\cdot t\rho (t\rho)^{-1}=\beta$.
Thus
\begin{eqnarray*}
\lefteqn{(st)\rho\cdot((st)\rho)^{-1} = }\\
&& = s\rho\cdot t\rho\cdot (t\rho)^{-1}\cdot (s\rho)^{-1} =  
s\rho\cdot \beta^{-1} \beta\cdot t\rho (t\rho)^{-1}\cdot (s\rho)^{-1}   \\
&& = s\rho\cdot \beta^{-1} \beta\cdot (s\rho)^{-1}
   = \alpha^{-1} \alpha.
\end{eqnarray*}
Hence $a\circ b = (\alpha, st, \gamma)$
is stable. 
Moreover, if $s'\in V(s)$, $t'\in V(t)$ in $S$ then 
Lemma \ref{tul}(\ref{tul2}),(\ref{tul3}) imply $(s's)\rho=\beta^{-1}\beta\le(tt')\rho$. 
Since $\beta^{-1}\beta$ is an idempotent $\rho$-class which is a completely simple subsemigroup of $S$ by assumption, 
we obtain that $(s'stt')\rho=\beta^{-1}\beta$ and $s's\irrr s'stt'$.
Since $\irr$ is a left congruence, this implies
$s\irrr stt'\irrr st$
whence
$a\irrr a\circ b$
follows.

(\ref{ujall3})\quad
The proof is similar to that of (\ref{ujall2}).
\end{proof}

An immediate consequence of this lemma is that $\widehat\irc$ is a 
locally inverse subsemigroupoid in $\irc$.
Furthermore, we have the following important property of stable arrows:

\begin{Lemma}\label{tomorit}
For every arrow $a\in\Arr\irc$, there is a unique stable arrow 
$b\in\Arr{\widehat\irc}$ such that $b\le a$.
\end{Lemma}

\begin{proof}
Let $a=(\alpha, s, \beta)\in\Arr\irc$. If $b=(\alpha, t, \beta)\in\Arr{\widehat\irc}$ with $b\le a$
then, by definition, we have $t\le s$ and $t\rho=\alpha^{-1}\beta$.
On the other hand, we see by Lemma~\ref{Sro-basic} that there exists a unique $t\in S$ such that
$t\le s$ and $t\rho=\alpha^{-1}\beta$. Lemma~\ref{tul} ensures that, in this case, we have
$b=(\alpha, t, \beta)\in\Arr{\widehat\irc}$, and the proof is complete.
\end{proof}

For any arrow $a$, denote by $\widehat a$ the unique stable arrow $b$ with
$b\le a$, and consider the graph morphism 
$\;\widehat{}\,\colon\irc\to\widehat\irc$ whose object mapping is identical and which
assigns $\widehat a$ to $a$ for every $a\in\Arr\irc$.
From now on, we put $\widehat A=\Arr{\widehat\irc}$.

Observe, that the graph morphism $\;\widehat{}\;$ respects the operations of $\irc$,
that is, it constitutes a binary semigroupoid morphism from $\irc$ onto $\widehat\irc$:

\begin{Lemma}\label{komp}
For any $a,b\in A$ with $\om a = \al b$, we have
\begin{enumerate}
\item \label{komp1}
$\widehat{\widehat a}=\widehat a$,
\item \label{komp2} 
$\widehat{a\circ b} = \widehat a\circ\widehat b$,
\item \label{komp3} 
$\widehat{(b\cw a)} = (\widehat b\cw\widehat a)$.
\end{enumerate}
\end{Lemma}

\begin{proof}
(\ref{komp1})\quad
Straightforward by definition.

(\ref{komp2})\quad
By definition, we have $\widehat a\le a,\ \widehat b\le b$, and so
$\widehat a\circ\widehat b\le a\circ b$. Moreover, $\widehat a\circ\widehat b$ 
is stable by Lemma~\ref{ujall}.
Therefore $\widehat a\circ\widehat b = \widehat{a\circ b}$ follows by 
Lemma~\ref{tomorit}.

(\ref{komp3})\quad
The proof is similar to that of (\ref{komp2}).
\end{proof}

By making use of the inverse unary operation ${}^\dagger$ on $\irc$, we extend the graph morphism 
$\,\widehat{}\,\colon\irc\to\widehat{\irc}$ to a binary semigroupoid morphism
from $\free2{\von{\irc}}$ in the way that we consider the graph morphism
$$\delta\colon\von{\irc}\to\widehat{\irc},\ a\delta=\widehat{a}\ \hbox{and}\ 
a'\delta=(\widehat{a})^\dagger\quad (a\in A),$$
and we define $\,\widehat{}\,\colon\free2{\von{\irc}}\to\widehat{\irc}$
to be the unique extension of $\delta$ to a binary semigroupoid morphism from $\free2{\von{\irc}}$ to 
$\widehat{\irc}$.
The restriction of $\;\widehat{}\;$ to $\widetilde{\irc}^+$, also denoted by $\;\widehat{}\;$, is obviously
a semigroupoid morphism.
In fact, it is also a binary semigroupoid morphism since each semigroup 
$\widehat{\irc}(\alpha,\alpha)\ (\alpha\in T)$ is completely simple.

Now we turn to proving that $\kappa$ is injective, that is, for every $s,t\in S$, the following implication holds
(see Result \ref{crit}):
\begin{equation}\label{to prove}
s\;\rho\;t\quad\hbox{and}\quad
\left(s\rho (s\rho)^{-1},s,s\rho\right)\;\thetat\;
\left(t\rho (t\rho)^{-1},t,t\rho\right) \quad\hbox{imply}\quad s=t.
\end{equation}
Let $s,t\in S$ with $s\,\rho\, t$ and, for brevity, put 
$a=\left(s\rho (s\rho)^{-1},s,s\rho\right)$ and
$b=\left(t\rho (t\rho)^{-1},t,t\rho\right)$.
Notice that $a,b$ are coterminal arrows in $\irc$, and, simultaneously, one-letter words in $\At^+$. 
Suppose that $a\,\thetat\,b$.
By Lemma~\ref{leirtaut}, there exists a 
$\ircs$-derivation
\begin{equation}\label{toprove}
a=\sero wn=b
\end{equation}
from $a$ to $b$.
We intend to prove that $a=b$ which clearly implies the equality $s=t$.
The crucial point in the proof is to describe the special features of the words of $\At^+$ 
appearing in such derivations.
Notice that derivation steps (T3b), (T4b) might introduce $\wedge$-letters which are not $\wedge$-loops. 
Consequently, $w_i\ (0<i<n)$ need not be a path in $\von{\irc}$.
The idea of our description of the words in such derivations is to indicate 
the breaking points of these kinds and their ranges
by pairs of brackets $\jobb{.}$ and $\bal{.}$, respectively.
For example, in the derivation
\[(a\wedge b),\ \ (a\wedge c)(a\wedge b),\ \ (a\wedge c)(d\wedge c)(a\wedge b),\]
where we apply rules (T3b) and (T4b), and
$(a\wedge b)$ is a $\wedge$-loop but $(a\wedge c),(d\wedge c)$ are not,
we indicate the breaking points as follows:
\[(a\wedge b),\ \ \jobb{(a\wedge c)}(a\wedge b),\ \ \jobb{(a\wedge c)\bal{(d\wedge c)}}(a\wedge b).\]

Now we introduce the set of words with brackets needed in our description.
Consider the free monoid
$(\At \cup \left\lbrace \bal{,},\jobb{,} \right\rbrace)^*$
where the empty word is denoted $\veps$,
and let $\Wt $ be its smallest subset which has the following four properties:
\begin{enumerate}[label=(\roman*)]
\item $\veps\in \Wt $;
\item $a\in \Wt $ for all $a\in \At $;
\item $w_1w_2\in \Wt $ for all $w_1,w_2\in \Wt $;
\item $\bal w,\jobb w\in\Wt $ for all $w\in \Wt\setminus \{\veps\}$.
\end{enumerate}
Notice that ${\At }^+\subseteq \Wt $.
In order to distinguish the elements of $\At^+$, called words, from those of $\Wt$, the latter will be called
{\em bracketed words}.
Moreover, the elements of $\At$ will be called {\em $\At$-letters}.
Recall that an $\At$-letter is either a letter of a $\wedge$-letter.
If $w\in \Wt$ then $\Iota w$ [$w\Tau$] denotes the first [last] element of
$\At \cup \left\lbrace \bal{,},\jobb{,} \right\rbrace$
appearing in $w$ (reading $w$ from the left to the right as a word in this alphabet).

For our later convenience, we introduce the notation $w{\downarrow}$ for the subword
of $w\in \Wt $ obtained from $w$ by deleting all brackets.
Clearly, $\veps{\downarrow}=\veps$ and if $w\neq\veps$ then $w{\downarrow}\in \At^+$.

Now we define three subsets $W_n$, $\Wr{n}$ and $\Wl{n}$ of $\Wt $ for every $n\in \mathbb{N}_0$. 
Simultaneously, we attach a (binary) path $\wp(w)\in\Arr{\widetilde{\irc}}^+$ to each element $w$ of these subsets.
If $\wp(w)$ is defined then we use $\whwp(w)$ to denote $\widehat{\wp(w)}$.
For technical reasons, we put $\wp(\veps)=\veps$ but let $\whwp(\veps)$ undefined.

Let $W_0=\Arr{\widetilde{\irc}}^+$, $W_0^\veps = W_0 \cup \{\veps \}$, and 
for any $w\in W_0$, define $\wp(w) = w$. 
Moreover, define
\[ \Wr{0} = \{ p(y \wedge x): p\in W_0^\veps,\ \alpha(y)\neq \omega(x),\ 
\hbox{and}\ \omega(p)=\alpha(y)\ \hbox{if}\ p\neq\veps \}, \]
and for any $w = p(y \wedge x)\in \Wr{0}$, let $\wp(w) = p(y\wedge y')$. 
By assumptions, this, indeed, belongs to $\Arr{\widetilde{\irc}}^+$. 
Similarly, let
\[ \Wl{0} = \{ (x \wedge y)p: p\in W_0^\veps,\ \alpha(x)\neq\omega(y),\
\hbox{and}\ \omega(y)=\alpha(p)\ \hbox{if}\ p\neq\veps \}, \]
and for any $w = (x \wedge y)p\in \Wl{0}$, let $\wp(w) = (y' \wedge y)p$.
Notice that $W_0\cup\Wr{0}\cup\Wl{0}\subseteq {\At }^+$.

Assume that $W_n$ [$\Wr{n},\ \Wl{n}$] is defined for some $n\in \mathbb{N}_0$, and a path 
$\wp(w)\in\Arr{\widetilde{\irc}}^+$ is assigned to each of its elements $w$.
For brevity, denote the set of all idempotent arrows of $\irc$ by $E$.
Define the set 
$W_{n+1}$ [$\Wr{n+1},\ \Wl{n+1}$] to consist of the bracketed words in $W_n$ [$\Wr{n},\ \Wl{n}$] and, additionally, of
all bracketed words $w\in \Wt $ of the form
\begin{equation}\label{generalword0}
w = p_0B_1C_1p_1B_2C_2\cdots B_kC_kp_k \quad (k\in \mathbb{N}),
\end{equation}
where the following conditions are satisfied:
\begin{enumerate}[label=(E\arabic*)]
\setcounter{enumi}{-1}
\item\label{E0}
$\ $
\begin{enumerate}[label=(E0\alph*)]
\item\label{E0a}
$p_1,\dots, p_{k-1}\in W_0$, $p_0\in W_0^\veps$ [$W_0^\veps,\ \Wl{0}$], 
$p_k\in W_0^\veps$ [$\Wr{0},\ W_0^\veps$], and
$\omega(p_{i-1})=\alpha(p_i)$ for every $i$ $(1\le i \le k)$,
\item\label{E0b}
$B_1C_1,\dots, B_kC_k\neq\veps$;
\end{enumerate}

\item\label{E1} 
for any $i$ $(1\le i \le k)$, we have
\begin{enumerate}[label=(E1\alph*)]
\item\label{E1a} 
$B_i = \bal{w_1}\bal{w_2}\cdots \bal{w_s}$, where $s\in \mathbb N_0$ and $w_j \in \Wr{n}$ $(1 \le j \le s)$, and
\item\label{E1b} 
for any $j$ $(1 \le j \le s)$, if $w_j\Tau =(y_j \wedge x_j)$ then
\begin{enumerate}[label=(E1b\roman*)]
\item\label{E1bi} 
$\whwp(w_j)\in E$ and $\widehat{y_j} \irrr \whwp(w_j)$, and
\item\label{E1bii} 
$\widehat{x_j} \irll\whwp(p_{i-1})$ (in particular, $p_0\neq \veps$ if $B_1\neq \veps$);
\end{enumerate}
\end{enumerate}

\item\label{E2} 
for any $i$ $(1\le i \le k)$, we have
\begin{enumerate}[label=(E2\alph*)]
\item\label{E2a} 
$C_i = \jobb{w_1}\jobb{w_2}\cdots \jobb{w_s}$, where $s\in \mathbb N_0$ and $w_j \in \Wl{n}$ $(1 \le j \le s)$, and
\item\label{E2b} 
for any $j$ $(1 \le j \le s)$, if $\Iota w_j =(x_j \wedge y_j)$ then
\begin{enumerate}[label=(E2b\roman*)]
\item\label{E2bi} 
$\whwp(w_j)\in E$ and $\widehat{y_j} \irll\whwp(w_j)$, and
\item\label{E2bii} 
$\widehat{x_j} \irrr \whwp(p_{i})$ (in particular, $p_k\neq \veps$ if $C_k\neq \veps$).
\end{enumerate}
\end{enumerate}
\end{enumerate}
For every $w\in W_{n+1}\setminus W_n$ [$\Wr{n+1}\setminus\Wr{n},\ \Wl{n+1}\setminus\Wl{n}$] of the form 
\eqref{generalword0}, 
define $\wp(w) = \wp(p_0)p_1\cdots p_{k-1}\wp(p_k)$.
Again, $\wp(w)$ is easily seen to belong to $\Arr{\widetilde{\irc}}^+$ by \ref{E0a} and by the definition of $\wp(w)$ for 
$w\in (\Wr{n+1}\setminus \Wr{n}) \cup (\Wl{n+1}\setminus \Wl{n})$.
The less trivial part to check is that $\wp(w)$ is non-empty if $w = p_0B_1C_1p_1\in W_{n+1}\setminus W_n$.
However, since either $B_1$ or $C_1$ is non-empty by \ref{E0b}, we get by \ref{E1b} or \ref{E2b} that $p_0 \neq \veps$ or $p_k\neq \veps$, respectively, whence $\wp(w)\neq \veps$ follows.

Finally, we define 
\[W = \bigcup_{n=0}^{\infty} W_n,\quad \Wr{} = \bigcup_{n=0}^{\infty} \Wr{n}\quad \hbox{and}\quad
\Wl{} = \bigcup_{n=0}^{\infty} \Wl{n}.\]

Alternatively, the bracketed words in $W\cup \Wr{}\cup \Wl{}$ can be characterized as follows.

\begin{Lemma}\label{word-char}
\begin{enumerate}
 \item \label{word-char1}
A bracketed word $w\in\Wt $ belongs to $W$ [\/$\Wr{},\ \Wl{}$] if and only if it is of the form
\begin{equation}\label{generalword}
w = p_0B_1C_1p_1B_2C_2\cdots B_kC_kp_k \quad (k\in \mathbb{N}_0),
\end{equation}
where either $k=0$ and $p_0\neq \veps$, or
the slightly modified versions of \ref{E0}--\ref{E2} are satisfied where $n$ is deleted from `$\Wr{n}$' and `$\Wl{n}$'
in \ref{E1a} and \ref{E2a}, respectively.
Moreover, this form of $w$ is uniquely determined.
 \item
For any bracketed word $w\in W\cup \Wr{}\cup \Wl{}$ of the form \eqref{generalword}, we have
\begin{equation*}
\wp(w) = \wp(p_0)p_1\cdots p_{k-1}\wp(p_k).
\end{equation*}
\end{enumerate}
\end{Lemma}

\begin{Rmk}
Notice that the description of the bracketed words belonging to $W\cup \Wr{}\cup \Wl{}$ which is formulated in 
Lemma \ref{word-char}(\ref{word-char1}) can be modified by deleting \ref{E0b} from the properties required, but then 
the form obtained is no more uniquely determined.
\end{Rmk}

Later on, when considering a bracketed word from $W\cup \Wr{}\cup \Wl{}$, we always consider it in its form described in
Lemma \ref{word-char}(\ref{word-char1}), but when checking whether a bracketed word belongs to $W\cup \Wr{}\cup \Wl{}$, we 
disregard checking property \ref{E0b}.

Notice that the set $W$ is self-dual in the sense that the reverse of each bracketed word 
of $W$ belongs to $W$. 
E.g., the reverse of the bracketed word $a\jobb{(b\wedge c)}(a'\wedge a)$ is $(a\wedge a')\bal{(c\wedge b)}a$
and vice versa.
Similarly, the sets $\Wr{}$ and $\Wl{}$ are dual to each other.

Besides bracketed words from $W\cup \Wr{}\cup \Wl{}$, we need also certain prefixes and suffixes of them which,
due to properties \ref{E1bii} and \ref{E2bii}, fail to belong to this set. 
Define $\Web$ [$\Wre$] to consist of all non-empty bracketed words $w$ of the form \eqref{generalword} where $p_0=\veps$, $B_1\neq
\veps$, and $w$ satisfies all conditions \ref{E0}--\ref{E2} for $W$ [$\Wr{}$] but \ref{E1bii} in case $i=1$.
Notice that $\wp(w)$ can be also defined for any $w\in \Web$ [$\Wre$] in the same way as it was done for bracketed words
in $W$ [$\Wr{}$], but this time $\wp(w)$ might be empty.
Clearly, we have $\wp(w)=\veps$ if and only if $k=1$ and $p_1=C_1=\veps$.
Dually, we define the set of bracketed words $\Wej$ [$\Wle$].

Given a bracketed word $w\in W\cup \Wr{}\cup \Wl{}$ of the form \eqref{generalword}, the following non-empty sections of $w$
are called {\em $\Wt $-suffixes of $w$ of type \ref{pref3}, \ref{pref1} and \ref{pref2}}, respectively:
\begin{enumerate}[label={\rm (\alph*)}]
\item\label{pref3}
$p_{i2}B_{i+1}C_{i+1}\cdots p_{k-1}B_kC_kp_k$ $(0\le i\le k)$
where $p_{i2}$ is a non-empty suffix of $p_i$,
\item\label{pref1}
$C_{i2}p_i\cdots B_kC_kp_k$ $(1\le i\le k)$
where $C_{i2}=\jobb{w_t}\jobb{w_{t+1}}\cdots\jobb{w_s}$ $(1\le t\le s)$ provided $C_i$ is of the form \ref{E2a}
with $s\neq 0$,
\item\label{pref2}
$B_{i2}C_i\cdots B_kC_kp_k$ $(1\le i\le k)$
where $B_{i2}=\bal{w_t}\bal{w_{t+1}}\cdots\bal{w_s}$ $(1\le t\le s)$ provided $B_i$ is of the form \ref{E1a}
with $s\neq 0$.
\end{enumerate}
It is obvious that a $\Wt $-suffix of $w$ is of the form \ref{pref3}, \ref{pref1} and \ref{pref2} if and only if
its first $\At$-letter belongs to $p_i$, $C_i$ and $B_i$, respectively.
The $\Wt $-prefixes of $w$ of type \ref{pref3}, \ref{pref1} and \ref{pref2} are defined dually.
The following statement is straightforward to check by definition.

\begin{Lemma}\label{cut}
Let $w\in W$ [\/$\Wr{},\ \Wl{}$], and let $v$ be a proper $\Wt $-suffix of $w$.
\begin{enumerate}
\item
If $v$ is of type \ref{pref3} or \ref{pref1} then $v\in W$ [\/$\Wr{},\ W$].
\item
If $v$ is of type \ref{pref2} then $v\in \Web$ [\/$\Wre,\ \Web$].
\end{enumerate}
Moreover, $\wp(v)=\veps$ if and only if $v$ is of the form $\bal{w_1}\bal{w_2}\cdots\bal{w_s}$ for some
$s\in\mathbb{N}$ and $w_1,w_2,\ldots,w_s\in \Wr{}$.
\end{Lemma}

The first two statements of the next lemma directly follow from the previous lemma.

\begin{Lemma}\label{deformedlemma-bracketed}
\begin{enumerate}[label=(\arabic*)]
\item \label{lai}
If $w\in \Wl{}$ then $\Iota w$ 
is a $\wedge$-letter which is not a $\wedge$-loop.
Moreover, if $w\not=\Iota w$ then $w$ has a proper $\Wt $-suffix which is either a path, 
or of the form $\bal{u}$ for some $u\in \Wr{}$.
The latter case occurs if and only if the last $\At$-letter of $w$
is a $\wedge$-letter which is not a $\wedge$-loop.
\item \label{bracketed}
If $w\in W$ then $w$ has a $\Wt $-suffix [\/$\Wt $-prefix] which is either a path, 
or of the form $\bal{u}$ [\/$\jobb{u}$] for some $u\in \Wr{}$ [\/$u\in \Wl{}$\/].
The latter case occurs if and only if the last [first] $\At$-letter of $w$
is a $\wedge$-letter which is not a $\wedge$-loop.
\item \label{laii} 
Let $\jobb{w}$ be a factor of a bracketed word in $W\cup \Wr{}\cup \Wl{}$ where
$w\in \Wl{}$ such that $w\neq\Iota w=(x\wedge y)$ $(x,y\in \overline{A})$, and the last $\At$-letter of $w$ is
$(b\wedge a)$ $(a,b\in \overline{A})$.
Then, independently of whether $\alpha(b)=\omega(a)$ or not,
we have $\widehat{a}\irll \widehat{y}\irll \whwp(w)\in E$ and $\omega(a)=\omega(y)=\alpha(\wp(w))=\omega(\wp(w))$.
Moreover, if $v$ is the $\Wt $-suffix of $w$ obtained from $w$ by deleting $\Iota w$, then
we have $v\in W\cup \Web$, and if $\wp(v)\neq \veps$ then 
$\widehat{y}\irll \whwp(v)\in E$ and $\omega(y)=\alpha(\wp(v))=\omega(\wp(v))$.
\end{enumerate}
\end{Lemma}

\begin{proof}
\ref{laii}
Assume that $w$ is of the form \eqref{generalword}.
If $\alpha(b)=\omega(a)$ then $(b\wedge a)=p_k\Tau$ and $p_k\neq \veps$.
Hence $\whwp(w)\irll \widehat{(b\wedge a)}\irll \widehat{a}$ follows.
Applying property \ref{E2bi}, we see that $\whwp(w)\in E$ and $\whwp(w)\irll \widehat{y}$ whence
$\widehat{a}\irll \widehat{y}\irll \whwp(w)\in E$ follows.
If $\alpha(b)\neq \omega(a)$ then the last factor in the form \ref{E1a} of $B_k$ is
$\bal{u}$ for some $u\in\Wr{}$ with $u\Tau=(b\wedge a)$.
This implies by \ref{E0a} that $p_{k-1}\neq\veps$ and $\wp(w)=\wp(p_0)\cdots \wp(p_{k-1})$, and so
$\whwp(p_{k-1})\irll \whwp(w)$ follows.
By property \ref{E1bii} of $u$ we deduce that $\widehat{a}\irll \whwp(p_{k-1})$, and
by property \ref{E2bi} of $w$ that $\whwp(w)\irll \widehat{y}$.
Thus we again obtain that $\widehat{a}\irll \widehat{y}\irll \whwp(w)\in E$.
In both subcases, this relation implies $\omega(a)=\omega(y)=\alpha(\wp(w))=\omega(\wp(w))$.

Turning to the second statement, first notice that Lemma \ref{cut} implies $v\in W\cup \Web$.
By definition, we have $\whwp(w)=\widehat{(y'\wedge y)}\whwp(v)$ where all three elements belong to a completely
simple subsemigroup of $S$.
This implies $\whwp(w)\irll\whwp(v)$.
Furthermore, we have seen in the first part of the proof that $\whwp(w)\in E$ and $\whwp(w)\irll \widehat{y}$.
Since $\widehat{(y'\wedge y)}\in E$ and $\widehat{(y'\wedge y)}\irll \widehat{y}$ also holds,
we deduce that $\whwp(v)\in E$ and $\whwp(v)\irll \widehat{y}$.
These relations imply $\omega(y)=\alpha(\wp(v))=\omega(\wp(v))$.
\end{proof}

An easy consequence of this lemma is that the subsets $W$, $\Wr{}$ and $\Wl{}$ of $\Wt $ are almost pairwise disjoint.

\begin{Cor}\label{e-r-l disjoint}
For the subsets $W$, $\Wr{}$, and $\Wl{}$ of $\Wt $, we have 
$W\cap(\Wr{}\cup \Wl{})=\emptyset$,
and $\Wr{}\cap \Wl{}$ is the set of all $\wedge$-letters which are not $\wedge$-loops. 
\end{Cor}

Let $w\in W\cup\Wr{}\cup\Wl{}$. 
We see by definition that if $\tilde u$ is any non-empty bracketed subword of $w$ then
two possibilities occur:
either $\tilde u$ is inside a pair of brackets $\bal{,}$ or $\jobb{,}$, 
or it is not.
In the first case, there exists a shortest section $v$ of $w$ 
such that $v$ contains $\tilde u$, and $v$ is either of the form $\bal{u}$ for some $u\in \Wr{}$ or of the form $\jobb{u}$ for some $u\in \Wl{}$.
We denote $u$ and $v$ by
$\sbs_w(\tilde u)$ and
$\sbbr_w(\tilde u)$, respectively.
In the second case, 
$\sbs_w(\tilde u)$ is defined to be $w$ and
$\sbbr_w(\tilde u)$ is undefined.

Now we are ready to return to proving the equality $a=b$ provided a $\ircs$-derivation (\ref{toprove}) is given 
from $a$ to $b$ where $a,b$ are coterminal arrows in $\irc$.
It suffices to show that, whenever $w,w^\wr\in {\widetilde A}^+$ such $w^\wr$ is obtained from $w$ by one of 
the derivation steps, 
and $\underline{w}\in W$ 
such that $w=\underline{w}{{\downarrow}}$, then
there exists a bracketed word $\underline{w^\wr}\in W$ such that $w^\wr=\underline{w^\wr}{\downarrow}$ and 
$\whwp(\underline{w})=\whwp(\underline{w^\wr})$.
For, if this holds, then we can choose $\underline{w_0}$ to be $a$, and we obtain $\underline{w_{i+1}}$ for $i=0,1,\ldots,n-1$ by induction such that $\whwp(\underline{w_i})=\whwp(\underline{w_{i+1}})$.
This implies $a=\wp(a)=\whwp(\underline{w_0})=\whwp(\underline{w_1})=\cdots
=\whwp(\underline{w_{n-1}})=\whwp(\underline{w_n})=\wp(b)=b$,
since $\underline{w_n}=\underline{w_n}{\downarrow}=w_n=b$.

In the rest of the section we verify the above statement for any derivation step.
In each subcase considered, the general scheme of the argument is as follows.
We consider $u=\sbs_{\underline{w}}(\tilde u)$ and $v=\sbbr_{\underline{w}}(\tilde u)$ 
for a bracketed subword $\tilde u$ of $\underline{w}$ 
such that $u{\downarrow}$ contains the section of $w$ involved in the derivation step, and define
$u^\wr\in \Wt $ such that the following conditions are satisfied:
\begin{itemize}
\item[(Q1)]
$u^\wr{\downarrow}$ is just the term obtained from $u{\downarrow}$ by the derivation step considered,
\item[(Q2)]
$u^\wr$ is of the form (\ref{generalword}), and if
$\Iota u=(x\wedge y)$ [$u\Tau=(x\wedge y)$] such that
$\alpha(x)\neq \omega(y)$, then
$\Iota u^\wr=(x^\wr\wedge y^\wr)$ [$u^\wr\Tau=(x^\wr\wedge y^\wr)$]
such that $\widehat{x^\wr}\irrr \widehat{x}$ and $\widehat{y^\wr}\irll \widehat{y}$,
\item[(Q3)]
$u^\wr$  has property \ref{E0a}, and we have 
$\whwp(u^\wr)=\whwp(u)$ [$\whwp(u^\wr)\irll\whwp(u)$, $\whwp(u^\wr)\irrr\whwp(u)$]
provided $\underline{w}=u$ [$v=\jobb{u}$, $v=\bal{u}$].
\item[(Q4)]
$u^\wr$ has properties \ref{E1}--\ref{E2}.
\end{itemize}
Notice that relations $\widehat{x^\wr}\irrr \widehat{x}$ and $\widehat{y^\wr}\irll \widehat{y}$ imply
$\alpha(x^\wr)=\alpha(x)$ and $\omega(y^\wr)=\omega(y)$.
Thus, by Corollary \ref{e-r-l disjoint}, (Q2)--(Q4) imply that $u^\wr\in W$, $u^\wr\in \Wr{}$ and $u^\wr\in \Wl{}$ if and only if $\underline{w}=u$, $v=\bal{u}$ and $v=\jobb{u}$, respectively.
Define $\underline{w^\wr}$ to be the bracketed word obtained from $\underline{w}$ by replacing the section $u$ by $u^\wr$.
To justify our approach, we have to verify that properties (Q1)--(Q4) imply $\underline{w^\wr}\in W$, $w^\wr=\underline{w^\wr}{\downarrow}$ and 
$\whwp(\underline{w})=\whwp(\underline{w^\wr})$.

Clearly, we have $\underline{w^\wr}=u^\wr$ if and only if $\underline{w}=u$,
and we have $\sbbr_{\underline{w^\wr}}(u^\wr)=\bal{u^\wr}$\,[$\jobb{u^\wr}$]
if and only if $v=\bal{u}$\,[$\jobb{u}$].
Moreover, property (Q2) implies that the factor $\bal{u^\wr}$ [$\jobb{u^\wr}$] of 
$\sbs_{\underline{w^\wr}}(\bal{u^\wr})$ [$\sbs_{\underline{w^\wr}}(\jobb{u^\wr})$] 
satisfies condition \ref{E1bii} [\ref{E2bii}],
since $\underline{w}\in W$, and so the factor $\bal{u}$ [$\jobb{u}$] of 
$\sbs_{\underline{w}}(\bal{u})$ [$\sbs_{\underline{w}}(\jobb{u})$] 
has property \ref{E1bii} [\ref{E2bii}].
All the details of properties \ref{E0}--\ref{E2} of $\underline{w^\wr}$ not checked in (Q2)--(Q4) are obviously 
inherited from those of $\underline{w}$.
This shows that $\underline{w^\wr}\in W$.
The equality $w^\wr=\underline{w^\wr}{\downarrow}$ is clear by (Q1) and by the definition of $\underline{w^\wr}$.
The equality $\whwp(\underline{w})=\whwp(\underline{w^\wr})$ is implied. 
For, if $\underline{w}\neq u$ then $\wp(\underline{w})$ is not affected by the changes done in $u$ to obtain $u^\wr$, and so $\wp(\underline{w^\wr})=\wp(\underline{w})$.
If $\underline{w}=u$ then we also have $u^\wr= \underline{w^\wr}$, and the equality follows from (Q3).

Note that, throughout the next proofs, (Q1) and (Q2) will be clear from the definition of $u^\wr$, and in a number of cases,
the same holds for (Q3).
Furthermore, most of the properties to be checked in (Q4) are inherited from the respective properties of $u$ and $\underline{w}$, or they are obvious by definition.
For example, (Q3) is clear if $\wp(u^\wr)=\wp(u)$, or condition \ref{E1bi} is trivially satisfied in case $w_j=w_j\Tau$.
It is also obvious that if $u$ is of the form \eqref{generalword} and $u^\wr$ is obtained from $u$ by deleting a factor 
$\bal{w_j}$ [$\jobb{w_j}$] (see \ref{E1a} [\ref{E2a}]) from $u$ then (Q2)--(Q4) are valid. 
In the proofs of the following propositions we concentrate on the properties being less trivial than these.

\begin{Prop}\label{stepS1}
Suppose that $w,w^\wr\in {\widetilde A}^+$ and we get $w^\wr$ from $w$ by a derivation step of one of the types
{\rm (S$j$a)}, {\rm (S$j$b)} for $j=1,21,22$ and {\rm (T5a)}, {\rm (T5b)}.
If $\underline{w}\in W$ such that $w=\underline{w}{\downarrow}$ then there exists 
$\underline{w^\wr}\in W$ such that $w^\wr=\underline{w^\wr}{\downarrow}$ and 
$\whwp(\underline{w})=\whwp(\underline{w^\wr})$.
\end{Prop}

\begin{proof}
First we consider the case of derivation steps (S22a) and (S22b).
By symmetry, we can assume that $w^\wr$ is obtained from $w$ by replacing 
either an occurrence of a $\wedge$-letter $(a\wedge y)$ by $(c\wedge y)$,
or an occurrence of a $\wedge$-letter $(c\wedge y)$ by $(a\wedge y)$,
where $y\in\Av$ and $a,c\in A$ such that $a\circ b=c$ for some $b\in A$.
By Lemma \ref{ujall}(\ref{ujall2}), this equality implies $\widehat{a}\irrr \widehat{c}$, 
and so $\al a=\al c$ follows.
Hence $(a\wedge y)$ is a $\wedge$-loop if and only if $(c\wedge y)$ is, and in this case, 
$\widehat{(a\wedge y)}=\widehat{(c\wedge y)}$.
If $(a\wedge y)$ is replaced by $(c\wedge y)$ then
put $u=\sbs_{\underline{w}}((a\wedge y))$, and consider its form \eqref{generalword}.
Define $u^\wr$ to be the bracketed word obtained from $u$
by replacing $(a\wedge y)$ by $(c\wedge y)$.
We see that $(a\wedge y)$ belongs to a section of $p_i$ for some $i$ $(0\le i\le k)$, 
and $(a\wedge y)$ is not a $\wedge$-loop if and only if either 
$i=0$, $(a\wedge y)=\Iota p_1=\Iota u$ and $\sbbr_{\underline{w}}((a\wedge y))=\jobb{u}$, or
$i=k$, $(a\wedge y)=p_k\Tau=u\Tau$ and $\sbbr_{\underline{w}}((a\wedge y))=\bal{u}$.
In these subcases, denote by $p_1^\wr$ and $p_k^\wr$ the words obtained from $p_1\in \Wl{0}$ and $p_k\in \Wr{0}$, respectively, by replacing $(a\wedge y)$ by $(c\wedge y)$.
By definition, we have $\wp(p_1)=\wp(p_1^\wr)$ in the first subcase, and since 
$\widehat{(a\wedge a')}\irrr \widehat{(c\wedge c')}$, we have $\whwp(p_k)\irrr\whwp(p_k^\wr)$ in the second subcase.
These observations imply properties (Q2)--(Q4).
The same argument applies if $(c\wedge y)$ is replaced by $(a\wedge y)$.

Turning to the rest of the derivations steps, 
denote by $p$ the section of $w$ modified by the derivation step, and by $q$ the word $p$ is replaced by in order 
to obtain $w^\wr$.
(Using the notation of Lemma \ref{leirtaut}, $p=s$, $q=t$ or $p=t$, $q=s$.)
With each derivation step considered, $p$ and $q$ are coterminal paths in $\Arr\widetilde{\irc}^+$ such that 
$\widehat{p}=\widehat{q}$.
Let $u=\sbs_{\underline{w}}(p)$ be of the form \eqref{generalword}.
Then $p$ is a section of $p_i$ for some $i$ $(0\le i\le k)$, and $p$ is not a prefix of $p_0$ 
[suffix of $p_k$] if $u\in \Wl{}$\,[$\Wr{}$].
Define $u^\wr$ to be the bracketed word obtained from $u$ by replacing the section $p$ of 
$u$ by $q$.
Thus  $u^\wr$ is obtained from $u$ by replacing a path section of $p_i$ by $q$.
Properties (Q3)--(Q4) are now easier to check than in case (S22a).
\end{proof}

The respective propositions for derivation steps (T3a), (T3b), (T4a), (T4b) are more complicated to prove.
However, (T3a) and (T3b) are duals of  (T4a) and (T4b), respectively, therefore we can restrict ourselves to
proving the latter ones.

\begin{Prop}\label{stepTb}
Suppose that $w,w^\wr\in {\widetilde A}^+$ and we get $w^\wr$ from $w$ by a derivation step of type
{\rm (T4b)}.
If $\underline{w}\in W$ such that $w=\underline{w}{\downarrow}$ then there exists 
$\underline{w^\wr}\in W$ such that $w^\wr=\underline{w^\wr}{\downarrow}$ and 
$\whwp(\underline{w})=\whwp(\underline{w^\wr})$.
\end{Prop}

\begin{proof}
Assume that an occurrence of a $\wedge$-letter $(y\wedge x)$ in $w$ is replaced by the word $(y\wedge x)(z\wedge x)$ 
where $x,y,z\in\overline{A}$.
Put $u=\sbs_{\underline{w}}((y\wedge x))$ and $v=\sbbr_{\underline{w}}(u)$.
If $(y\wedge x)$ is not a $\wedge$-loop then we have either $v=\jobb{u}\ (u\in \Wl{})$, or $v=\bal{u}\ (u\in \Wr{})$.

First we suppose that $(y\wedge x)$ is a $\wedge$-loop, or
$(y\wedge x)$ is not a $\wedge$-loop, and $v=\jobb{u}$, $u\in \Wl{}$, $\Iota u=(y\wedge x)$.
If $u$ is of form (\ref{generalword}) then in these cases, $(y\wedge x)$ is in $p_i$ for some $i\ (1\le i\le k$), and
if $(y\wedge x)$ is not a $\wedge$-loop then necessarily $i=0$ and $(y\wedge x)=\Iota p_0$.
Thus we have $p_i=p_{i1}(y\wedge x)p_{i2}$ for some $i$ where $p_{i1}$ and $p_{i2}$ are (possibly empty) paths,
$p_{01}$ being necessarily empty if $(y\wedge x)$ is not a $\wedge$-loop.
Define $u^\wr$ to be the bracketed word 
obtained from $u$ by replacing the $\wedge$-letter $(y\wedge x)$ by the bracketed word
$(y\wedge x)(z\wedge x)$ or $(y\wedge x)\bal{(z\wedge x)}$
according to whether $(z\wedge x)$ is a $\wedge$-loop or not.

If $\alpha(z)=\omega(x)$ then $u^\wr$ is obtained from $u$ such that $p_i$ is replaced by 
$p_i^\wr=p_{i1}(y\wedge x)(z\wedge x)p_{i2}$, and section $p_i^\wr$ of $u^\wr$ 
belongs to $W_0$ if $p_i\in W_0$, and belongs to $\Wl{0}$ if $i=0$ and $p_0\in \Wl{0}$.
Moreover, $\widehat{(z\wedge x)}$ is an idempotent $\irl$-related to $\widehat{x}$, therefore
$\whwp(p_i)=\whwp(p^\wr_i)$.
Similarly to the end of the proof of Proposition \ref{stepS1}, this equality implies properties (Q3)--(Q4).
If $(z\wedge x)$ is not a $\wedge$-loop then we have
\begin{equation}\label{T4i}
u^\wr=p_0\cdots p_{i-1}B_iC_ip_{i1}(y\wedge x)\bal{(z\wedge x)}p_{i2}B_{i+1}C_{i+1}p_{i+1}\cdots p_k.
\end{equation}
To verify (Q4), it suffices to show that the factor $\bal{(z\wedge x)}$ and those of $B_{i+1}$ satisfy condition \ref{E1bii}. 
The former holds since $\whwp(p_{i1}(y\wedge x))\irll \widehat{x}$.
To see the latter, we recall the respective relation between $p_i$ and $B_{i+1}$ in $u$ and the facts that if 
$p_{i2}\neq \veps$ then $\whwp(p_i)\irll \widehat{p_{i2}}$, and if
$p_{i2}= \veps$ then $\whwp(p_i)=\whwp(p_{i-1}(y\wedge x))\irll \widehat{x}$.

Now suppose that $v=\bal{u}\ (u\in \Wr{})$, and so $u\Tau=(y\wedge x)$.
Consider the section $u_+=\sbs_{\underline{w}}(v)$
of $\underline{w}$, and suppose that it is of the form \eqref{generalword}.
Then $v$ is a factor of $B_i$ for some $i$ $(1\le i\le k)$, therefore
$B_i$ is of the form $\bal{u_{-m}}\cdots\bal{u_{-1}}\bal{u}\bal{u_1}\cdots\bal{u_n}\ (m,n\in\mathbb{N}_0)$
for some bracketed words $u_j\in \Wr{}\ (-m\le j\le n,\ j\neq 0)$.
Define $u_+^\wr$ to be the bracketed word obtained from $u_+$ by replacing $\bal{u}$ by 
$\bal{u}(z\wedge x)$ if $(z\wedge x)$ is a $\wedge$-loop, and by $\bal{u}\bal{(z\wedge x)}$ otherwise.
Thus
\begin{equation*}
u_+^\wr=p_0\cdots p_{i-1}\bal{u_{-m}}\cdots\bal{u_{-1}}\bal{u}(z\wedge x)\bal{u_1}\cdots\bal{u_n}C_ip_i\cdots p_k
\end{equation*}
and
\begin{equation*}
u_+^\wr=p_0\cdots p_{i-1}\bal{u_{-m}}\cdots\bal{u_{-1}}\bal{u}\bal{(z\wedge x)}\bal{u_1}\cdots\bal{u_n}C_ip_i\cdots p_k,
\end{equation*}
respectively, in the two subcases.
In the second subcase, (Q3) is clear. 
Since $u\Tau=(y\wedge x)$ implies by property \ref{E1bii} of $u_+$ that $\widehat{x}\irll \whwp(p_{i-1})$, 
we immediately obtain that the new factor $\bal{(z\wedge x)}$ satisfies condition \ref{E1bii},
and so (Q4) also follows.
In the first subcase, where $(z\wedge x)$ is a new path factor, 
the relation $\widehat{x}\irll \whwp(p_{i-1})$, seen above, implies $\widehat{p_{i-1}}\widehat{(z\wedge x)}=\widehat{p_{i-1}}$ and
$\omega(p_{i-1})=\alpha(z)=\omega(x)$, since $\widehat{(z\wedge x)}$ is idempotent. This verifies (Q3). 
Moreover, we obtain that the factors $\bal{u_j}\ (1\le j\le n)$
satisfy condition \ref{E1bii} whence (Q4) follows.
\end{proof}

\begin{Prop}\label{stepTa}
Suppose that $w,w^\wr\in {\widetilde A}^+$ and we get $w^\wr$ from $w$ by a derivation step of type
{\rm (T4a)}.
If $\underline{w}\in W$ such that $w=\underline{w}{\downarrow}$ then there exists 
$\underline{w^\wr}\in W$ such that $w^\wr=\underline{w^\wr}{\downarrow}$ and 
$\whwp(\underline{w})=\whwp(\underline{w^\wr})$.
\end{Prop}

\begin{proof}
Assume that an occurrence of a section $(y\wedge x)(z\wedge x)$ of $w$ is replaced by $(y\wedge x)$ where 
$x,y,z\in\overline{A}$.
Denote $\sbs_{\underline{w}}((y\wedge x))$, $\sbs_{\underline{w}}((z\wedge x))$ and
$\sbs_{\underline{w}}((y\wedge x)(z\wedge x))$ by $u_1$, $u_2$ and $u$, respectively.
Clearly, $u_1$ and $u_2$ are sections of $u$, and each can be equal to $u$ or can be a proper subsection of $u$.
We proceed by distinguishing the four cases obtained in this way.

Case $u=u_1=u_2$.\quad
If $u$ is of the form \eqref{generalword} then $(y\wedge x)(z\wedge x)$ is a section of $p_i$ for some $i\ (0\le i\le k)$.
This implies that $\omega(x)=\alpha(z)$, and so $(z\wedge x)$ is a $\wedge$-loop.
If $(y\wedge x)$ is not a $\wedge$-loop then \ref{E0a} implies $i=0$, $p_0\in \Wl{0}$ and $\Iota p_0=(y\wedge x)$.
Define $p_i^\wr$ and $u^\wr$ to be the bracketed words obtained from $p_i$ and $u$, respectively,
by deleting $(z\wedge x)$.
Property (Q3) follows from the fact that, if $(y\wedge x)$ is a $\wedge$-loop then $\widehat{(y\wedge x)}$ and $\widehat{(z\wedge x)}$,
if $(y\wedge x)$ is not a $\wedge$-loop then $\widehat{(x'\wedge x)}$ and $\widehat{(z\wedge x)}$
are $\mathcal{L}$-related idempotents, and so we have
$\widehat{(y\wedge x)}\widehat{(z\wedge x)}=\widehat{(y\wedge x)}$ and
$\widehat{(x'\wedge x)}\widehat{(z\wedge x)}=\widehat{(x'\wedge x)}$,
respectively.
To check (Q4), it suffices to observe that $\whwp{(p_i)}\irll \whwp(p_i^\wr)$ by the former equalities 
if $(y\wedge x)(z\wedge x)$ is a suffix of $p_i$, and by the equality $p_i\Tau=p_i^\wr\Tau$ otherwise.

Case $u=u_1\neq u_2$.\quad
Put $v_2=\sbbr_{\underline{w}}(u_2)$ where we have $v_2=\bal{u_2}$ and $u_2\in \Wr{}$, 
or $v_2=\jobb{u_2}$ and $u_2\in \Wl{}$.
Assume that $u$ is of the form \eqref{generalword}.
Then $(y\wedge x)=p_{i-1}\Tau$ and $(z\wedge x)$ is the first $\At$-letter of the bracketed word $B_iC_i$ for some $i\ (1\le i\le k)$, 
and if $(y\wedge x)$ is not a $\wedge$-loop then $i=1$ and $p_0=(y\wedge x)$.

If $B_i=\veps$ then $(z\wedge x)$ is contained in the first factor of $C_i$ of the form \ref{E2a}.
Therefore the first factor of $C_i$ is $v_2=\jobb{u_2}$ where $u_2\in \Wl{}$ and $\Iota u_2=(z\wedge x)$, and so 
$\alpha(z)\neq \omega(x)$.
Notice that $\omega(x)=\omega(p_{i-1})=\alpha(p_i)$ and $\widehat{z}\irrr \whwp{(p_i)}$
by properties \ref{E0a} and \ref{E2bii} of $u$,   
and the latter relation implies $\alpha(z)=\alpha(p_i)$.
Hence we obtain that $\alpha(z)=\omega(x)$, a contradiction.

If $B_i\neq\veps$ then $(z\wedge x)$ is the first $\At$-letter of the first factor $\bal{w_1}$ of $B_i$
of the form \ref{E1a} where $w_1\in \Wr{}$.
By the dual of Lemma \ref{deformedlemma-bracketed}\ref{lai} we see that either $\alpha(z)\neq \omega(x)$ and 
$u_2=w_1=(z\wedge x)$, or $(z\wedge x)$ is a $\wedge$-loop and $u_2=w_1$, or else
$\alpha(z)\neq \omega(x)$, $u_2\in \Wl{}$ with $\Iota u_2=(z\wedge x)$, and 
$v_2=\jobb{u_2}$ is a $\Wt $-prefix $w_1$. 
Now we consider these subcases separately.

If $\alpha(z)\neq \omega(x)$ and $u_2=w_1=(z\wedge x)$ then define $u^\wr$ to be the bracketed word obtained from $u$ by deleting
the factor $\bal{w_1}$ of $B_i$.
This obviously fulfils all the requirements.

Now consider the subcase where $(z\wedge x)$ is a $\wedge$-loop, i.e., $\al z=\om x$, and $u_2=w_1$.
If $w_1\Tau=(a\wedge b)$ then the dual of Lemma \ref{deformedlemma-bracketed}\ref{laii} implies $\al z=\al a$, and
\ref{E1bii} ensures $\widehat{b}\irll \whwp(p_{i-1})\irll \widehat{x}$
since $\wp(p_{i-1})\Tau$ is either $(y\wedge x)$ or $(x'\wedge x)$, depending on whether
$(y\wedge x)$ is a $\wedge$-loop or not.
This implies $\omega(b)=\omega(x)$ whence we obtain $\alpha(a)=\omega(b)$, a contradiction.

Finally, let $\alpha(z)\neq \omega(x)$, $u_2\in \Wl{}$ such that $\Iota u_2=(z\wedge x)$ and 
$v_2=\jobb{u_2}$ is a $\Wt $-prefix of $w_1$.
Therefore we have $u_2=(z\wedge x)u_{22}$ and $w_1=\jobb{u_2}w_{12}$ for some 
$\Wt $-suffix $u_{22}$ and $w_{12}$ of $u_2$ and $w_1$, respectively, whence
$u_{22}\in W\cup\Web$ and $w_{12}\in \Wr{}$ by Lemma \ref{cut}.
This allows us to define $u^\wr$ so that the section 
$(y\wedge x)\bal{w_1}=(y\wedge x)\bal{\jobb{(z\wedge x)u_{22}}w_{12}}$ of $u$, where $(y\wedge x)=p_{i-1}\Tau$,
is replaced by $(y\wedge x)u_{22}\bal{w_{12}}$.
Since $u_2=(z\wedge x)u_{22}\in \Wl{}$, it is easy to see by definition that
$(y\wedge x)u_{22}\in W$ or $\Wl{}$ depending on whether $(y\wedge x)$ is a $\wedge$-loop or not.
This implies that $u^\wr$ is of the form \eqref{generalword}.
Applying Lemma \ref{deformedlemma-bracketed}\ref{laii} for $u_2$, we obtain that if
$\wp(u_{22})\neq\veps$ then $\widehat{x}\irll \whwp(u_{22})\in E$.
Hence $\whwp((y\wedge x)u_{22})=\whwp({(y\wedge x)})\whwp(u_{22})=\whwp((y\wedge x))$ follows,
and this implies $\whwp(p_{i-1}u_{22})=\whwp(p_{i-1})$, and so (Q3) holds for $u^\wr$.
In order to check (Q4) for $u^\wr$, it suffices to verify that the factor $\bal{w_{12}}$ satisfies \ref{E1b}.
Since $w_{12}\Tau=w_1\Tau$ and $\wp(w_{12})=\wp(w_1)$, it is straightforward from property \ref{E1bi} of 
$\bal{w_1}$ in $u$ that the same property is valid for $\bal{w_{12}}$ in $u^\wr$.
Similarly, these equalities combined with $\whwp(p_{i-1}u_{22})=\whwp(p_{i-1})$
allow us to see that property \ref{E1bii} of $\bal{w_1}$ in $u$ implies the same property of $\bal{w_{12}}$ 
in $u^\wr$.

Case $u=u_2\neq u_1$.\quad 
Assume that $u$ is of the form \eqref{generalword}.  
Then $(z\wedge x)=\Iota p_i$ and $(y\wedge x)$ is the last $\At$-letter in the bracketed word $B_iC_i$ for some 
$i$ $(1\leq i \leq k)$. 
Furthermore, if $(z\wedge x)$ is not a $\wedge$-loop then $i=k$, $(z\wedge x)=p_k$,
and by Corollary \ref{e-r-l disjoint}, $u\in\Wr{}$.

First we examine the subcase, where $(z\wedge x)$ is a $\wedge$-loop.
If $C_i\neq \veps$ then $s\in \mathbb N$ in \ref{E2a}, and $(y\wedge x)$ is the last $\At$-letter of $w_s$.
Assume that $\Iota w_s=(a\wedge b)$ where $\alpha(a)\neq\omega(b)$. 
Property \ref{E2bii} of $w_s$ implies that $\widehat{a}\irrr\whwp(p_i)\irrr \widehat{(z\wedge x)}\irrr \widehat{z}$ 
whence $\alpha(a)=\alpha(z)$ follows.
If $w_s=\Iota w_s$ then $b=x$ and $\omega(b)=\omega(x)$ are obvious.
If $w_s\neq \Iota w_s$ then we see by Lemma \ref{deformedlemma-bracketed}\ref{laii} that
$\omega(b)=\omega(x)$.
Combining these equalities we obtain $\alpha(z)=\alpha(a)\neq \omega(b)=\omega(x)$, a contradiction.

Let us assume now that $C_i= \veps$, and so $(y\wedge x)$ is the last $\At$-letter of $B_i$.
In the form \ref{E1a} of $B_i$, we have $s\in \mathbb N$ and $(y\wedge x)=w_s\Tau$.
Since $(z\wedge x)=\Iota p_i$, we have $p_i=(z\wedge x)p_{i2}$ where $p_{i2}\in W_0^\veps$ or
$i=k$ and $p_{i2}\in \Wr{0}$.
Define
\[ u^\wr=p_0 \cdots p_{i-1}B_ip_{i2}B_{i+1}C_{i+1}p_{i+1}\cdots p_k. \]
Since $(z\wedge x)$ is a $\wedge$-loop and we have $\whwp(p_{i-1}) \irll \widehat{x}$ 
in $u$ by property \ref{E1bii} of $w_s$, we get $\whwp(p_{i-1})\widehat{(z\wedge x)}=\whwp(p_{i-1})$.
Hence 
$\whwp(p_{i-1}p_i)=\whwp(p_{i-1}(z\wedge x)p_{i2})=\whwp(p_{i-1})\widehat{(z\wedge x)}\whwp(p_{i2})=
\whwp(p_{i-1})\whwp(p_{i2})=\whwp(p_{i-1}p_{i2})$ also if $p_{i2}\neq\veps$, and (Q3) follows.
If $p_{i2}\neq\veps$ then the relation $\widehat{p_i}\irll \widehat{p_{i2}}$ implies that the factors of
$B_{i+1}$ fulfil condition \ref{E1bii} in $u^\wr$ since they do in $u$.
If $p_{i2}=\veps$ then the same follows by observing that $\widehat{p_i}=\widehat{(z\wedge x)}\irll \widehat{x}$ in $u$, 
and so $\widehat{p_i}\irll \widehat{p_{i-1}}$.

Secondly, consider the subcase where $(z\wedge x)$ is not a $\wedge$-loop. 
As we have seen above, $u$ is necessarily in $\Wr{}$, and $p_k=(z\wedge x)$ in its form \eqref{generalword}.
If $C_k=\veps$ then $(y\wedge x)$ is the last $\At$-letter in $B_k$,
and so in its form \ref{E1a} we have $s\in \mathbb N$ and $(y\wedge x)=w_s\Tau$. 
By \ref{E0a} it follows that $\omega(p_{k-1})=\alpha(p_k)=\alpha(z)$.
Also, by applying \ref{E1bii} for $w_s$, we obtain that $\widehat{x}\irll\whwp(p_{k-1})=\widehat{p_{k-1}}$, 
which implies $\omega(p_{k-1})=\omega(x)$. 
Hence we conclude $\alpha(z)=\omega(x)$, which contradicts the assumption that 
$(z\wedge x)$ is not a $\wedge$-loop.

If $C_k\neq\veps$ then $(y\wedge x)$ is the last $\At$-letter in $C_k$ where it is of the form \ref{E2a} with 
$s\in \mathbb N$.
Assume that $\Iota w_s=(a\wedge b)$ where, by definition, $\alpha(a)\neq \omega(b)$.
By property \ref{E2bii} of $w_s$ in $u$ we see that $\widehat{a}\irrr \whwp(p_k) \irrr \widehat{z}$.
If $\Iota w_s=w_s$ then we have $w_s=(a\wedge b)=(y\wedge x)$, and so 
$\widehat{y}=\widehat{a}\irrr \widehat{z}$ follows. 
In this case, let us define
$u^\wr=p_0B_1C_1p_1\cdots p_{k-1}B_k\jobb{w_1}\cdots\jobb{w_{s-1}}(y\wedge x).$
Obviously, the relation $\widehat{y}\irrr \widehat{z}$ implies properties (Q2) and (Q4), the rest being even more
straightforward.

Now consider the subcase $\Iota w_s\neq w_s$.
Then $w_s=(a\wedge b)w_{s2}$ such that $w_{s2}$ is the $\Wt $-suffix of $w_s\in \Wl{}$ obtained 
by deleting $\Iota w_s=(a\wedge b)$, and so 
the last $\At$-letter of $w_{s2}$ is $(y\wedge x)$ 
and 
$w_{s2}\in W\cup \Web$ by Lemma \ref{cut}.
Recall the relation  $\widehat{a}\irrr \widehat{z}$ from the previous paragraph, and notice that
$\widehat{b}\irll \widehat{x}$ follows by applying Lemma \ref{deformedlemma-bracketed}\ref{laii} for $w_s$ in $u$.
Consider the section 
$v=\sbs_{\overline{w}}(\bal{u})$
of $\overline{w}$, and let its form \eqref{generalword} be
\[v=\ulp_0\ulB_1\ulC_1\ulp_1\ulB_2\ulC_2\ulp_2\cdots \ulp_{l-1}\ulB_l\ulC_l\ulp_l\quad (l\in\mathbb{N}).\]
Then $\bal{u}$ is a factor of $\ulB_i$ for some $i\ (1\leq i\leq l)$, 
more precisely, we have 
\[\ulB_i=\bal{\ulw_1}\cdots\bal{\ulw_{j-1}}\bal{u}\bal{\ulw_{j+1}}\cdots\bal{\ulw_t}\quad (t\in\mathbb{N}),\]
where $\ulw_m\in\Wr{}\ (1\le m\le t,\ m\neq j)$.
For brevity, put
\[\ulB_{i1}=\bal{\ulw_1}\cdots\bal{\ulw_{j-1}}\quad\hbox{and}\quad
\ulB_{i2}=\bal{\ulw_{j+1}}\cdots\bal{\ulw_t},\]
and so we have $\ulB_i=\ulB_{i1}\bal{u}\ulB_{i2}$.
Define
\[ u_0^\wr=p_0B_1C_1p_1\cdots p_{k-1}B_k\jobb{w_1}\cdots\jobb{w_{s-1}}(a\wedge b) \]
and
\[ v^\wr=\ulp_1\ulB_1\ulC_1\ulp_1\cdots\ulp_{i-1}\ulB_{i1}\bal{u_0^\wr}w_{s2}\ulB_{i2}\ulC_i\ulp_i\cdots\ulB_l\ulC_l\ulp_l.\]
Notice that $u_0^\wr\in\Wr{}$ which directly follows from the facts that $u\in\Wr{}$ and
$\widehat{a}\irrr \widehat{z},\ \widehat{b}\irll \widehat{x}$.
Since $w_{s2}\in W\cup \Web$ the bracketed word
$v^\wr$ is of the form \eqref{generalword}, and conditions (Q1) and (Q2) are clearly satisfied by $v$ and $v^\wr$.
If $\wp(w_{s2})=\veps$ then (Q3) is also obvious. 
If $\wp(w_{s2})\neq \veps$ then, applying Lemma \ref{deformedlemma-bracketed}\ref{laii} for $w_s$ in $u$, we see that
$\widehat{x}\irll \widehat{b}\irll \wp(w_{s2})\in E$, and,
by using \ref{E1bii} for the factor $\bal{u}$ of $\ulB_i$, 
we obtain that $\whwp(\ulp_{i-1})\irll \widehat{x}$.
Hence we conclude that $\whwp(\ulp_{i-1})\whwp(w_{s2})=\whwp(\ulp_{i-1})$, 
and (Q3) holds also if $\wp(w_{s2})\neq \veps$.
Moreover, these observations combined with the respective properties of $v$ imply most items of property (Q4).
It remains to observe that if $w_{s2}$ has a non-empty $\Wt $-prefix of the form 
$\mathring{B}_1=\bal{\mathring{w}_1}\bal{\mathring{w}_2}\cdots \bal{\mathring{w}_n}$ where 
$\mathring{w}_m\in\Wr{}$ and $\mathring{w}_m\Tau=(y_m\wedge x_m)$ $(1\le m\le n)$, in particular, if $\wp(w_{s2})=\veps$,
then $\widehat{x_m}\irll \whwp(\ulp_{i-1})$.
For, $\widehat{x_m}\irll \widehat{b}$ follows from the property \ref{E2bi} of $w_s$ in $u$.

Case $u\neq u_1,u_2$.
Observe that in this case $\sbbr(u_1)$ and  $\sbbr(u_2)$ are disjoint. 
Therefore, considering $u$ in the form \eqref{generalword}, each of
$u_1$ and $u_2$ is in a factor $\bal{w_j}$ of some $B_i$ (see \ref{E1a}) or in a factor $\jobb{w_j}$ of some $C_i$
(see \ref{E2a}).
First assume that $(y\wedge x)$ is the last $\At$-letter of $B_i$ and $(z\wedge x)$ is the first $\At$-letter of $C_i$
for some $i\ (1\le i\le k)$, and so
$(y\wedge x)=w_s\Tau$ where $\bal{w_s}$ is the last factor of $B_i$ of the form \ref{E1a}, and 
$(z\wedge x)=\Iota\ulw_1$ where $\jobb{\ulw_1}$ is the first factor of $C_i$ of the form \ref{E2a}.
This implies that $\alpha(y), \alpha(z)\neq \omega(x)$.
By \ref{E0a}, we have $\omega(p_{i-1})=\alpha(p_i)$, and by \ref{E1bii}, we have
$\whwp(p_{i-1})\irll \widehat{x}$, whence $\omega(p_{i-1})=\omega(x)$ follows. 
Similarly, $\widehat{z}\irrr \whwp(p_i)$ by \ref{E2bii}, and so $\alpha(z)=\alpha(p_i)$. 
Hence we obtain $\alpha(z)=\omega(x)$, a contradiction.

Now assume that both $u_1$ and $u_2$ are in $B_i$ for some $i\ (1\leq i\leq k)$.
Then, considering $B_i$ in the form \ref{E1a}, there exists $j\ (1<j\le s)$ such that 
$(y\wedge x)$ is the last $\At$-letter of $w_{j-1}$ and $(z\wedge x)$ is the first $\At$-letter of $w_j$.
Thus $(y\wedge x)=w_{j-1}\Tau$ and $\alpha(y)\neq \omega(x)$ follow, 
and we have either $w_j=(z\wedge x)$ and $\alpha(z)\neq \omega(x)$, or $w_j\neq w_j\Tau$. 
If $w_j=(z\wedge x)$ then define $u^\wr$ to be the bracketed word obtained from $u$ by replacing $B_i$ by
$B^\wr_i$ where $B^\wr_i$ is obtained from $B_i$ by deleting the factor $\bal{w_j}$.
It is straightforward that (Q1)--(Q4) hold.

Now we turn to the subcase $w_j\neq w_j\Tau$.
Then $(z\wedge x)$ is the first $\At$-letter of $w_j$, and $w_j\Tau=(a\wedge b)$
with $\alpha(a)\neq \omega(b)$.
We obtain by the dual of Lemma \ref{deformedlemma-bracketed}\ref{laii} that
$\alpha(z)=\alpha(a)$, and we see by  applying \ref{E1bii} for $w_{j-1}$ and $w_j$ that 
$\widehat{x} \irll \whwp(p_{i-1})\irll \widehat{b}$, and so $\omega(x)=\omega(b)$. 
This implies $\alpha(z)\neq \omega(x)$, and we deduce by Lemma \ref{deformedlemma-bracketed}\ref{lai}
that $u_2\in\Wl{}$, $\Iota u_2=(z\wedge x)$, and $\jobb{u_2}$ is a prefix of $w_j$.
Hence $w_j=\jobb{u_2}w_{j2}$ and $u_2=(z\wedge x)u_{22}$ where
$w_{j2}$ and $u_{22}$, if $u_{22}\neq\veps$, are $\Wt $-suffixes of $w_j$ and $u_2$, respectively,
and we have $w_{j2}\Tau=w_j\Tau$ and $\wp(w_{j2})=\wp(w_j)$.
Since $w_{j2}$ is of type (a) or (b), we see by Lemma \ref{cut} that $w_{j2}\in \Wr{}$ and 
$u_{22}\in W\cup \Web$.
Let us define $u^\wr$ to be the bracketed word obtained from $u$ by replacing $B_i$ by
\[B^\wr_i=\bal{w_1}\cdots \bal{w_{j-1}}u_{22}\bal{w_{j2}}\bal{w_{j+1}}\cdots \bal{w_s}.\]
Since $u_{22}\in W\cup \Web$, $u^\wr$ is of the form \eqref{generalword} and (Q2) holds.
If $\wp(u_{22})=\veps$ then (Q3) is obvious.
In the opposite case, we apply Lemma \ref{deformedlemma-bracketed}\ref{laii} for $u_2$ to see that
$\widehat{x}\irll \whwp(u_{22})\in E$.
By property \ref{E1bii} of
the factor $\bal{w_{j-1}}$ of $B_i$ we have $\widehat{x}\irll \whwp(p_{i-1})$, and so $\whwp(p_{i-1})\whwp(u_{22})=\whwp(p_{i-1})$. 
Hence (Q3) follows also in case $\wp(u_{22})\neq\veps$. 
In order to check (Q4),
assume that $u_{22}$, if non-empty, is of the form
$u_{22}=\bal{\mathring{w}_1}\bal{\mathring{w}_2}\cdots \bal{\mathring{w}_n}u_{22}^{\mathrm{ab}}$ where
$n\in\mathbb{N}_0$, $\mathring{w}_m\in\Wr{}$ $(1\le m\le n)$, and
$u_{22}^{\mathrm{ab}}$ is empty or is the longest $\Wt $-suffix of $u_{22}$ of type (a) or (b).
By Lemma \ref{cut}, either $u_{22}^{\mathrm{ab}}=\wp(u_{22}^{\mathrm{ab}})=\veps$, or 
$\wp(u_{22})=\wp(u_{22}^{\mathrm{ab}})\neq\veps$.
Therefore it suffices to show
that \ref{E1bii} is satisfied by the following factors of $B_i^\wr$:
$\bal{\mathring{w}_m}$ $(1\le m\le n)$, provided $n\neq 0$, and 
$\bal{w_{j2}},\bal{w_{j+1}},\ldots, \bal{w_s}$, provided $\wp(u_{22})\neq\veps$.
If $\mathring{w}_m\Tau=(y_m\wedge x_m)$ $(1\le m\le n)$ 
then property \ref{E1bii} of the former factors in $u_2$ implies 
$\widehat{x_m}\irll \whwp((z\wedge x))\irll \widehat{x}$.
Similarly, the same property of the factors of $B_i$ in $u$ ensures that 
$\widehat{x}\irll \widehat{b}\irll \widehat{b_r}$ if $w_r\Tau=(a_r\wedge b_r)$ $(j<r\le s)$,
since $w_{j-1}\Tau=(y\wedge x)$ and $w_j\Tau=(a\wedge b)=w_{j2}\Tau$.
This verifies property (Q4) because it is seen above that $\widehat{x}\irll \whwp(p_{i-1})$, and 
if $u_{22}\neq\veps$ then also $\widehat{x}\irll \whwp(u_{22})$.

Finally, assume that both $u_1$ and $u_2$ are in $C_i$ for some $1\leq i\leq k$. 
Then, considering $C_i$ in the form \ref{E2a}, there exists $j\ (1<j\le s)$ such that 
$(y\wedge x)$ is the last $\At$-letter of $w_{j-1}$ and $(z\wedge x)$ is the first $\At$-letter of $w_j$.
Hence we obtain that $(z\wedge x)=\Iota w_j$ with $\alpha(z)\neq \omega(x)$, and so
$w_j=(z\wedge x)w_{j2}$ where $w_{j2}$, if non-empty, is a $\Wt $-suffix of $w_j$. 
Define $u^\wr$ to be the bracketed word obtained from $u$ by replacing $C_i$ by
\[C^\wr_i=\jobb{w_1}\cdots \jobb{w_{j-2}}\jobb{w_{j-1}w_{j2}}\jobb{w_{j+1}}\cdots \bal{w_s}.\]
All we have to show is that $w_{j-1}w_{j2}\in \Wl{}$ and 
the factor $\jobb{w_{j-1}w_{j2}}$ of $C^\wr_i$ has property \ref{E2bi}.
For, \ref{E2bii} follows from the same property of $C_i$ due to the equality $\Iota(w_{j-1}w_{j2})=\Iota w_{j-1}$.
By the same argument applied in the previous paragraph for $u_2$ and $u_{22}$, we can deduce that
if $w_{j2}\neq\veps$ then $w_{j2}\in \Web$, and if $\wp(w_{j2})\neq\veps$ then $\widehat{x}\irll \whwp(w_{j2})\in E$
and $\omega(x)=\alpha(\wp(w_{j2}))$.
Put $\Iota w_{j-1}=(a\wedge b)$.
If $w_{j-1}=\Iota w_{j-1}$ then $a=y,\ b=x$, and
if $w_{j-1}\neq\Iota w_{j-1}$ then Lemma \ref{deformedlemma-bracketed}\ref{laii}
implies that
$\widehat{x}\irll \widehat{b}\irll \whwp(w_{j-1})\in E$ and $\omega(x)=\omega(b)=\omega(\wp(w_{j-1}))$.
Therefore, whether $w_{j-1}=\Iota w_{j-1}$ or not, $\omega(\wp(w_{j-1}))=\alpha(\wp(w_{j2}))$ follows 
if $\wp(w_{j2})\neq\veps$, 
and we can deduce that $w_{j-1}w_{j2}$ is of the form \eqref{generalword} where \ref{E0a} holds with 
$p_0\in \Wl{0}$, and so (Q3) is satisfied.
To verify property \ref{E1} of $w_{j-1}w_{j2}$, we again refer to the argument on $u_2$ and $u_{22}$ in the previous paragraph which shows in our present case that if
$w_{j2}=\bal{\mathring{w}_1}\bal{\mathring{w}_2}\cdots \bal{\mathring{w}_n}w_{j2}^{\mathrm{ab}}$ where
$n\in\mathbb{N}_0$, $\mathring{w}_m\in\Wr{}$ with $\mathring{w}_m\Tau=(y_m\wedge x_m)$ $(1\le m\le n)$, and
$w_{j2}^{\mathrm{ab}}$ is empty or is the longest $\Wt $-suffix of $w_{j2}$ of type (a) or (b),
then 
$\widehat{x_m}\irll \whwp((z\wedge x))\irll \widehat{x}$.
Combining this with the previous relations
$\widehat{x}\irll \whwp(w_{j2})\in E$ if $\wp(w_{j2})\neq\veps$ and
$\widehat{x}\irll \widehat{b}\irll \whwp(w_{j-1})\in E$,
we obtain that $\whwp(w_{j-1}w_{j2})=\whwp(w_{j-1})\whwp(w_{j2})=\whwp(w_{j-1})$, and so
\ref{E1} holds in $w_{j-1}w_{j2}$.
Since $w_{j-1}\in \Wl{}$ and $w_{j2}\in \Web$, \ref{E2} is clearly fulfilled in $w_{j-1}w_{j2}$, thus we have shown
that $w_{j-1}w_{j2}\in \Wl{}$.
\end{proof}

\section{Concluding remarks}

The main result of \cite{8} proves that, given a group variety $\iru$, if $S$ is an inverse semigroup and $\rho$ an idempotent separating congruence on $S$ such that the idempotent classes of $\rho$ belong to $\iru$ 
then the extension $\Sro$ is embeddable in a $\lambda$-semidirect product extension of a member of $\iru$ by $S/\rho$.

The question naturally arises whether Theorem \ref{main} can be strengthened so that the variety of all completely 
simple semigroups be replaced by any variety of completely simple semigroups.

\begin{Problem}
For which varieties $\irv$ of completely simple semigroups is it true that
if $S$ is an $E$-solid locally inverse semigroup and $\rho$ an inverse semigroup congruence on $S$
such that the idempotent classes of $\rho$ belong to $\irv$ 
then the extension $\Sro$ is embeddable in a $\lambda$-semidirect product extension of a
member of $\irv$ by $S/\rho$?
\end{Problem}

Note that in the special case where $\irv$ is the variety of rectangular bands, the answer is affirmative.
The approach applied in the proof of Theorem \ref{main} works, and the technical details are significantly simpler
(no $\wedge$ operation is needed, the invariant congruence corresponding to the variety of rectangular bands is easy
to handle).
Thus the following result yields.

\begin{Prop}
A regular semigroup is a generalized inverse semigroup if and only if it is embeddable in a $\lambda$-semidirect product
of a rectangular band by an inverse semigroup.
\end{Prop}

\end{document}